\RequirePackage{fix-cm}

\documentclass[smallextended]{svjour3}       % onecolumn (second format)
\smartqed  
\usepackage{mysty}

\makeatletter 
\let\reftagform@=\tagform@
\def\tagform@#1{\maketag@@@{(\ignorespaces\textcolor{black}{#1}\unskip\@@italiccorr)}}
\newcommand{\iref}[1]{\textup{\reftagform@{\tcr{\ref{#1}}}}}
\makeatother

\begin{document}

\title{Partial Smoothness, Subdifferentials and Set-valued Operator
}

\author{Ziqi Qin \and
        Jingwei Liang %etc.
}


\institute{Ziqi Qin \at
              School of Mathematical Sciences \& Institute of Natural Sciences, Shanghai Jiao Tong University, 200240 Shanghai China.\\
              \email{zqqin58@sjtu.edu.cn}           %  \\
           \and
           Jingwei Liang \at
           School of Mathematical Sciences \& Institute of Natural Sciences, Shanghai Jiao Tong University, 200240 Shanghai China.\\
              \email{jingwei.liang@sjtu.edu.cn}   
}

\date{Received: date / Accepted: date}

\maketitle

\begin{abstract}
Over the past decades, the concept ``partial smoothness'' has been serving as a powerful  tool in several fields involving nonsmooth analysis, such as nonsmooth optimization, inverse problems and operation research, etc. 
The essence of partial smoothness is that it builds an elegant connection between the optimization variable and the objective function value through the subdifferential. 
Identifiability is the most appealing property of partial smoothness, as locally it allows us to conduct much finer or even sharp analysis, such as linear convergence or sensitivity analysis. 
However, currently the identifiability relies on non-degeneracy condition and exact dual convergence, which limits the potential application of partial smoothness. 
In this paper, we provide an alternative characterization of partial smoothness through only subdifferentials. This new perspective enables us to establish stronger identification results, explain identification under degeneracy and non-vanishing error. 
Moreover, we can generalize this new characterization to set-valued operators, and provide a complement definition of partly smooth operator proposed in \cite{lewis2022partial}. 
\end{abstract}

\section{Introduction}\label{sec:introduction}

Nonsmoothness is at the heart of modern optimization, the advances in nonsmooth analysis have greatly impacted the development of (first-order) optimization algorithms, and spread to many other fields including operation research, inverse problems, signal/image processing, data science and statistics, to name a few. 
Along the study of nonsmoothness, an important topic is exploring the underlying smooth structure, especially around the optimal solutions. Such results not only yield high-level geometrical understanding of the optimization problem, but also stimulate the development of efficient numerical schemes.  
This direction of research dates back to linear programming where complementary slackness captures the smoothness inside constraints which is the principle of active set methods. Wright extends this idea to smooth identifiable surfaces \cite{Wright-IdentSurf}, which is further extended to smooth Riemannian manifold by Lewis and Hare \cite{hare2004identifying} and encoded as ``partial smoothness''. Another similar branch of work is called ``$\calU\calV$ decomposition'' \cite{lemarechal2000U}, see also \cite{mifflin2004VU}.

An important property of the smooth structure is called ``identifiability''. This means that given an optimization problem $\min_{x} f(x)$ whose optimal solution(s) are contained in a smooth Riemannian manifold $\calM$, a sequence, \eg $\seq{\xk}$, is generated to find an optimal solution $\xsol \in \calM$, then for all $k$ large enough, there holds
\beq\label{eq:identification}
\xk \in \calM .
\eeq
% \todo{$\xsol\rightarrow\xbar$, $v\rightarrow u$?}
Under partial smoothness \cite{LewisPartlySmooth}, identification \eqref{eq:identification} can be guaranteed if $\xsol$ is non-degenerate (\ie $\ubar \in \ri\Pa{\partial f(\xsol)}$ where $\ri\Pa{\cdot}$ denotes relative interior), and there exists dual sequence $\seq{\uk}$ associated to $\seq{\xk}$ satisfying
\beq\label{eq:uk-to-ubar}
\partial f(\xk) \ni \uk \to \ubar  .
\eeq
Since first proposed in \cite{LewisPartlySmooth}, the theoretical framework of partial smoothness is established over a series of work \cite{hare2004identifying,LewisPartlyTiltHessian,drusvyatskiy2014optimality}, and finds successful applications in several fields including sensitivity analysis \cite{mordukhovich1992sensitivity}, local linear convergence analysis of first-order optimization algorithms \cite{liang2014local,liang2017activity} and design efficient numerical schemes \cite{poon2019trajectory}, etc.

Though very powerful, partial smoothness has several limitations, such as requiring non-degeneracy condition for identification, dual vector convergence and restricted to functions. 
Attempts are made in the literature to overcome them, for example in \cite{fadili2018sensitivity}, for functions with stratifiable structure, the authors study the identification under degeneracy conditions. In \cite{drusvyatskiy2014optimality}, partial smoothness in terms of normal cone operator is studied. Later on in \cite{lewis2022partial}, partly smooth operator is proposed together with a so-called ``constant rank'' property which characterizes the local property of the operator's graph. 
Unfortunately, these work fail to provide satisfactory solutions to the limitations of partial smoothness.

\paragraph{Contribution}
Motivated by the limitations of partial smoothness, in this paper we propose a novel alternative definition (see Definition \ref{dfn:ps-svo}) of partial smoothness from functions to set-valued operators. Under this new definition, we are able to provide a finer local geometric characterization of non-degeneracy condition. From this, we have the following contributions
\begin{enumerate}[label={\rm (\roman{*})}]
    \item In Definition \ref{dfn:ps-svo}, we propose a new definition of partial smoothness of set-valued operator, which is a complement of the constant-rank property of \cite{lewis2022partial}, calculus rules are also provided. 
    
    \item We provide a new geometric characterization for partial smoothness. Take subdifferential $\partial f$ of partly smooth function $f$ for example,  in Proposition \ref{prop:local-union}, we show that the linear span of the local union $\calU = \cup_{\ba{ x\in\calM, \norm{x-\xbar}<\varepsilon} }(x+\partial f(x))$ is the whole space; Moreover, there holds $\xbar+\ubar \in \inte(\calU)$, which directly leads to the identifiability of the manifold; see Proposition \ref{prop:local-union} and Theorem \ref{thm:ps-ident}. 
    
    \item Owing to the new local characterization, we are able to show that identification still happens even if condition \eqref{eq:uk-to-ubar} fails, given that essentially $\uk$ is close enough to $\ubar$; see Theorem \ref{thm:ps-ident}. In Corollary \ref{prop:iden-degenerate}, we also discuss the identifiability without non-degenerate condition. 
    
    \item Under proper condition, identification only occurs for large enough $k$. In Proposition \ref{prop:step}, we provide an smallest possible estimation of $k$ after which identification happens. 
\end{enumerate}
To illustrate these theoretical results, we apply our findings to variational inequalities and mini-batch stochastic gradient descent, demonstrating how our approach can effectively explain practical problems.

\paragraph*{Paper organization}
The following of the paper is organized as: in Section \ref{sec:pre} we collect some necessary preliminarily materials. Section \ref{sec:partly-smooth} introduces the concept of partial smoothness and outlines the fundamental properties of partial smooth operators. In Section \ref{sec:identifiability}, we discuss our main results regarding the identifiability of partly smooth operators, offering a new perspective on partial smoothness and extending identifiable conditions to more general cases. Finally, in Section \ref{sec:applications}, we provide several applications to illustrate the identification of partly smooth operators with concrete experimental results.

\section{Notations and preliminaries}\label{sec:pre}

We denote $\bbN$ the set of non-negative integers and $k \in \bbN$ the index. $\bar{\bbR} = \bbR\cup\Ba{+\infty}$ denotes the extended real line. 
$\Rn,\Rm$ are the standard finite dimensional Euclidean spaces, equipped with inner product $\iprod{\cdot}{\cdot}$ and norm $\norm{\cdot}$. 

Define $\ball{r}{x}$ the ball with center $x$ and radius $r$. 
Let $\cS \subset \Rn$ be a nonempty closed convex set, if its interior exists, we denote it as $\inte(\cS)$, the boundary is denoted as $\bdy(\cS)$; If the interior does not exist, then $\ri(\cS )$ denotes its relative interior, and $\rbd(\cS)$ denotes the relative boundary. We denote $\Aff(\cS )$ its affine hull, $\Lin(\cS ) = \Aff(\cS ) - x, \forall x\in\cS$ the subspace parallel to it. 
The distance function between a point $z\in\Rn$ and $\cS$ is defined by
\[
\dist(z, \cS) \eqdef \inf_{x\in\cS} \norm{z-x} .
\]
Denote $\proj_{\cS }$ the orthogonal projector onto $\cS $ and $\normal{\cS}$ its normal cone operator.

Given a proper, lower semi-continuous (l.s.c.) function $f:\Rn \to \bar{\bbR}$ and $x\in \dom(f)$, the Fr\'echet (or regular) subdifferential $\partial^{F} f(x)$ of $f$ at $x$, is the set of vectors $v\in\bbR^{n}$ that satisfies 
$$
\liminf_{z\to x ,\,  z \neq x}~ \sfrac{1}{\norm{x-z}} \pa{f(z) - f(x) - \iprod{v}{z-x}} \geq 0 .
$$ 
%\[
%\liminf_{z\to x ,\,  z \neq x} \sfrac{1}{\norm{x-z}} \Pa{R(z) - R(x) - \iprod{v}{z-x}} \geq 0  .
%\]
If $x \notin \dom(f)$, then $\partial^{F} f(x) = \emptyset$. 
The limiting-subdifferential (or simply subdifferential) of $f$ at $x$, written as $\partial f(x)$, is defined as 
$$
\partial f(x) \eqdef \bBa{ v \in \bbR^{n} : \exists \xk \to x , f(\xk) \to f(x) ,~   \vk \in \partial^{F} f(\xk) \to v } .
$$ 

Denote $\dom(\partial f) \eqdef \Ba{x\in\bbR^n: \partial f(x) \neq \emptyset}$, an element in $\partial f$ is called the subgradient. 

Let $S\subseteq \Rn$ be a set, though indicator function $\iota_{S}(x)$ of the set, the Fr\'echet (or regular) normal cone to $S$ at a point $x\in S$ is defined by $\sN^{F}_{S}(\xbar) = \partial^{F} \iota_{S}(\xbar)$, and the (limiting) normal cone is $\sN_{S}(\xbar) = \partial \iota_{S}(\xbar)$. Both normal cones are defined to be empty for $x\notin S$. $S$ is said to be (Clarke) regular at $\xbar$ if it is locally closed at $\xbar$ and the two normal cones agree. Given a function $f:\Rn \to \bar{\bbR}$, its epi-graph is defined as ${\rm epi}(f) \eqdef \Ba{ (x,t) \in \Rn \times \bbR \mid f(x) \leq t } $, then $f$ is regular at $\xbar$ if its epi-graph is regular at $(\xbar, f(\xbar))$; in this case $\partial^F f(\xbar) = \partial f(\xbar)$.

\subsection{Set-valued operators and motononicity}

Let $A: \Rn \setvalued \Rn$ be a set-valued operator, the domain of $A$ is $\dom(A) = \ba{ x\in \Rn \mid A(x) \neq \emptyset }$, the range of $A$ is $\ran(A) = \ba{u \in \Rn \mid \exists x \in \Rn : u \in A(x) }$, the graph of $A$ is the set $\gra(A)=\ba{ (x,u) \in \Rn \times \Rn \mid u \in A(x) }$, and its zeros set is $\zer(A) = \ba{ x\in\Rn \mid 0 \in A(x) }$. For any $u \in A(x)$, we call it a  {\it dual vector} of $x$.

\subsubsection{Continuity of set-valued operators}

Different from single-valued analysis, there are two distinct concepts for the continuity of set-valued operators: lower semi-continuity and upper semi-continuity.\footnote{Lower and upper semi-continuity are also referred to as inner and outer semi-continuity. Here we adopt these notion from the book \cite{aubin2009set}, and these notions should not be confused with the lower and upper semi-continuity of real-valued functions.} The definitions below are from \cite{aubin2009set}.

\begin{definition}[Lower semi-continuity]\label{def:lower-sc}
A set-valued operator $A: \Rn \setvalued \Rn$ is called \emph{lower} semi-continuous at $\xbar \in \dom(A)$ if and only if for any $\ubar \in A(\xbar)$ and any sequence $\xk \in \dom(A)$ converging to $\xbar$, there exists a sequence of dual vectors $\uk \in A(\xk)$ converging to $\ubar$. It is moreover said to be lower semi-continuous if it is lower semi-continuous at every point in $x \in \dom(A)$. 
\end{definition}

In the single-valued case, the above definition is equivalent to: for any open subset $\calU \subset \Rn$ such that $\calU \cap A(\xbar) \neq \emptyset$,
\[
\exists \eta > 0 ~~~\st~~~ \forall x \in \ball{\eta}{\xbar},~ A(x) \cap \calU \neq \emptyset .
\]

\begin{definition}[Upper semi-continuity]\label{def:upper-sc}
A set-valued operator $A: \Rn \setvalued \Rm$ is called \emph{upper} semi-continuous at $\xbar \in \dom(A)$ if and only if for any for any neighborhood $\calU$ of $A(\xbar)$,
\[
\exists \eta > 0 ~~~\st~~~ \forall x \in \ball{\eta}{\xbar},~ A(x) \subset \calU .
\]
It is moreover said to be upper semi-continuous if it is upper semi-continuous at every point in $x \in \dom(A)$. 
\end{definition}

Alternatively, upper semi-continuity means that given a sequence $\xk \to \xbar$ with $\uk\in A(\xk)$ and $\uk\to\ubar$, then there holds $\ubar \in A(\xbar)$. 
When $A(x)$ is compact, $A$ is upper semi-continuous at $x$ if and only if 
\[
\forall \varepsilon > 0,~ \exists \eta > 0 ~~~ \st ~~~ \forall x \in \ball{\eta}{\xbar},~~ A(x) \subset \ball{\varepsilon}{A(\xbar)}  .
\]
Noted that this definition is a natural extension of the definition of a continuous single-valued operator.

%\vskip3mm

\begin{lemma}[{\cite[Theorem~8.6, Proposition~8.7]{rockafellar1998variational}}]
Let $f:\Rn \to \bar{\bbR}$, and $x\in\dom(f)$ with $f(x)$ finite, then both $\partial^{F} f(x), \partial f(x$) are closed, with $\partial^{F} f(x)$ being convex and $\partial^{F} f(x) \subset \partial f(x)$. 
%If $f$ is lower semi-continuous at $x$, it is (subdifferentially) regular at $x$ if and only if $\partial^{F} f(x) = \partial f(x)$.
If $f$ is proper and lower semi-continuous functions, then its (limiting) subdifferential is upper semi-continuous. %, and in general not lower semi-continuous. 
\end{lemma}

\subsubsection{Monotoncity of set-valued operators} 

%An important class of set-valued operators is the monotone operators. 
Below we provide some basic definition of monotonicity, and refer to \cite{bauschke2011convex} for dedicated discussions.

\begin{definition}%[Monotone operator]
A set-valued operator $A : \Rn \setvalued \Rn$ is {monotone} if
\beqn
\iprod{x-y}{u-v} \geq 0 ,~~~ {\forall (x,u)\in\gra(A), ~ \forall (y,v)\in\gra(A)}    .
\eeqn
It is moreover called {maximally monotone} if its graph $\gra(A)$ can not be contained in the graph of any other monotone operators. 

\end{definition}

An important source of monotone operators is the subdifferential of proper lower semi-continuous (l.s.c.) and convex functions.

\begin{lemma}[{\cite[Baillon--Haddad theorem]{baillon1977quelques}}]
\label{lemma:prop_subdiff}
Let $f: \Rn \to \bbR\cup \{+\infty\}$ be proper l.s.c and convex, then $\partial f$ is maximally monotone. 
\end{lemma}

Let $A : \Rn \setvalued \Rn$ be set-valued and $\gamma > 0$, its \emph{resolvent}, denoted by $\rslt_{\gamma A}$, is defined by
\beq\label{eq:resolvent}
\rslt_{\gamma A} \eqdef \pa{\Id + \gamma A}^{-1}  .
\eeq
Note that here maximal monotonicity is not imposed here, hence given a point $x\in\Rn$, $\rslt_{\gamma A}(x)$ is a set which can be empty. When $A$ is maximally monotone, the following result ensures $\rslt_{\gamma A}(x)$ is a singleton.

\begin{lemma}[{\cite[Corollary 23.11]{bauschke2011convex}}]
\label{lemma:resolvent-maximally-monotone}
Let $A:\Rn \setvalued \Rn$ be maximally monotone and $\gamma > 0$, then $\rslt_{\gamma A}:\Rn \to \Rn$ is single-valued, maximally monotone and firmly nonexpansive.
\end{lemma}

When $A=\partial f$ is the subdifferential of a proper l.s.c. function $f$, computing the resolvent of $A$ is equivalent to evaluating the proximal operator of $f$, that is
\[
(\Id + \gamma \partial f)^{-1}(\cdot)
= \prox_{\gamma f}(\cdot)
\eqdef \argmin_{x} \gamma f(x) + \sfrac{1}{2}\norm{x-\cdot}^2 .
\]

When $f$ is only proper and l.s.c., the proximal operator $\prox_{\gamma f}(\cdot)$ is also set-valued. When $f$ is moreover convex, $\prox_{\gamma f}(\cdot)$ becomes well-defined and single-valued as stated in Lemma \ref{lemma:resolvent-maximally-monotone}.

\subsection{Prox-regularity}

To obtain well-posedness of the proximal operator in the nonconvex setting, certain regularity property is needed for $f$, which is described in the definition below.

\begin{definition}[Prox-Regularity \cite{poliquin1996prox}]\label{def:prox-regular}
A function $f : \Rn \to \bar{\bbR}$ is prox-regular at a point $\xbar\in\dom(f)$ for a subgradient $\ubar \in \partial f(\xbar)$, if there exist $r > 0$ and $\varepsilon > 0$ such that
\beq\label{eq:prox-regular}
f(x') > f(x) + \iprod{u}{x' - x} - \sfrac{r}{2} \|x' - x\|^2 
\eeq
whenever $\|x' - \xbar\| < \varepsilon$ and $\|x - \xbar\| < \varepsilon$, with $x'\neq x$ and $|f(x) - f(\xbar)| < \varepsilon$, for $\|u - \ubar\| < \varepsilon$ with $u \in \partial f(x)$.
Moreover, $f$ is prox-regular at $\xbar$ if it is prox-regular at $\xbar$ for every $\vbar \in \partial f(\xbar)$.  
\end{definition}

Prox-regularity is a local characterization of the function with restrictions to both $x$ and $u\in\partial f$, it allows to guarantee local single-valuedness of $\prox_{\gamma f}$ under proper choice of $\gamma$. 
Given $\varepsilon > 0$, the local \emph{$f$-attentive neighborhood} \cite[Definition 3.1]{poliquin1996prox} is defined by $\Ba{x \in \Rn \mid \|x - \xbar\| < \varepsilon, \ |f(x) - f(\xbar)| < \varepsilon }$. 
From this, we can further define the $f$-attentive $\varepsilon$-localization, denoted as $T(x)$, of the subdifferential $\partial f$ at $(\xbar, \ubar)$, 
\[
T(x)
= 
\left\{
\begin{aligned}
    \Ba{u\in\partial f(x) \mid \norm{u-\ubar}\leq \varepsilon } \quad & {\rm for~} \|x - \xbar\| < \varepsilon {\rm ~and~} \ |f(x) - f(\xbar)| < \varepsilon \\
    \emptyset \quad & o.w.
\end{aligned}
\right.
\]

\begin{lemma}[{\cite[Theorem 3.2]{poliquin1996prox}}]
Let $f$ be locally l.s.c. at $\xbar$. If $f$ is prox-regular at $\xbar$ for $\ubar$, then there exists $\varepsilon > 0$ and $r > 0$, with
$
f(x) > f(\xbar) + \iprod{\ubar}{x - \xbar} - \frac{r}{2} \|x - \xbar\|^2 , ~~ 0< \norm{x-\xbar} < \varepsilon
$, 
such that the mapping $(\Id + r\partial f)^{-1}$ has the following {\it single-valuedness} property near $\zbar = \xbar + r \ubar$: if $\norm{z-\zbar} < \varepsilon$, and if for $i=0,1$, one has
\[
x_i \in (\Id + r\partial f)^{-1}(z) \quad{\rm with}\quad \norm{x_i-\xbar}<\varepsilon,~~ \abs{f(x_i)-f(\xbar)}<\varepsilon ,
\]
then necessarily $x_0=x_1$. 

\end{lemma}

The vector $\ubar$ is called a proximal subgradient of $f$ at $\xbar$. 
Equivalently, the above result indicates the $f$-attentive $\varepsilon$-localization $T$ satisfies $T + r \Id$ is monotone, hence the resolvent $(T + r \Id)^{-1}$ is single-valued. 
Extending the prox-regularity of function to general set-valued operators via subdifferential, we need the following localization.

\begin{definition}[Localization]\label{def:localization}
Let $A: \Rn \setvalued \Rn$ be a set-valued operator. For $\varepsilon > 0$, the $\varepsilon$-localization, denoted by $A_{\varepsilon}$, of $A$ around $(\xbar, \ubar) \in \gra(A)$ is defined by
\beq\label{eq:eps-localization}
A_{\varepsilon}(x)
=
\left\{
\begin{aligned}
&\ba{u \in A(x) : \norm{u-\ubar} < \varepsilon}, && \textrm{if}~~ \norm{x-\xbar} < \varepsilon  , \\
&\emptyset,  && \textrm{otherwise.}
\end{aligned}
\right.
\eeq
\end{definition}

Now we define the local regularity of resolvent. 

\begin{definition}[Resolvent-regularity]\label{def:prox-reg}
A set-valued operator $A: \Rn \setvalued \Rn$ is called \emph{resolvent-regular} at $\xbar \in \dom(A)$ for $\ubar \in A(\xbar)$, if there exist $\varepsilon > 0$ and $r > 0$ such that: $r\Id + A_{\varepsilon}$ is monotone where $A_{\varepsilon}$ is the $\varepsilon$-localization of $A$ around $(\xbar,\ubar)$, and for each $\gamma \in ]0, 1/r[$ there is a neighborhood $\Rn$ of $\xbar$ such that the operator $\Jj_{\gamma \Ae}$ is single-valued and continuous on $\Rn$ and 
\[
\Jj_{\gamma \Ae} = \pa{\Id + \gamma A_{\varepsilon}}^{-1}  .
\]
Furthermore, $A$ is resolvent-regular at $\xbar$ if it is resolvent-regular at $\xbar$ for every $\ubar \in A(\xbar)$. 
\end{definition}

\section{Partial smoothness}\label{sec:partly-smooth}

In this section, we provide the extension of partial smoothness from functions to operators, accompanied with calculus rules. 
% \cite{LewisPartlySmooth,drusvyatskiy2014optimality}.
Let $\Mm \subset \Rn$ be a $C^p$-smooth manifold with $p\geq1$, given $x \in \Rn$, denote $\tanSp{\Mm}{x}$ the tangent space of $\Mm$ at $x$.

\subsection{Partly smooth functions}

We start with recalling the concept of partly smooth functions and identifiability, as discussed in \cite{LewisPartlySmooth,hare2004identifying}.

 \begin{definition}[Partly smooth function]\label{dfn:psf}
 A function $f:\Rn \to \bar{\bbR}$ is partly smooth at a point $\xbar$ relative to a set $\calM$ containing $\xbar$ if $\calM$ is a $C^p$-smooth manifold around $\xbar$ and 
 \begin{enumerate}[label={\rm (\roman{*})}]
 \item {\pa{\bf Smoothness}} \label{psf:smoothness} 
 $f$ restricted to $\calM$ is $C^p$-smooth around $\xbar$;
 \item {\pa{\bf Prox-regularity}} \label{psf:regularity} 
$f$ is regular at all $x\in\calM$ near $\xbar$, $\partial f(x) \neq \emptyset$. % and prox-regular at $\xbar$;
 \item {\pa{\bf Sharpness}} \label{psf:sharpness} 
$\tanSp{\calM}{\xbar} = \LinHull\pa{\partial f(\xbar)}^\perp$;
 \item {\pa{\bf Continuity}} \label{psf:continuity} 
 $\partial f$ restricted to $\calM$ is continuous at $\xbar$.
 \end{enumerate}
 \end{definition}

Through subdifferential, partial smoothness builds an elegant connection between functions and the manifold $\calM$, and moreover establishes the identifiability of $\calM$. Loosely speaking, identifiability implies that given a sequence $\seq{\xk}$ that converges to $\xbar$, under suitable conditions, $\xk$ will find $\calM$ first and then converge to $\xbar$ along the manifold. Rigorously, we have the following result from \cite{hare2004identifying}.

\begin{proposition}[{\cite[Theorem 5.3]{hare2004identifying}}]\label{prop:iden-function}
Let function $f:\Rn\to\bar{\bbR}$ be $C^p$-partly smooth ($p\geq 2$) at the point $\xbar$ relative to the manifold $\calM$, and prox-regular there with $\ubar \in \ri\Pa{\partial f(\xbar)}$. Suppose $\xk\to\xbar$ and $f(\xk) \to f(\xbar)$. Then for all $k$ large enough, there holds
\[
\xk \in \calM 
\]
if and only if 
\[
\dist\Pa{\ubar, \partial f(\xk) } \to 0 .
\]
\end{proposition}

\begin{remark}\label{rmk:iden_func}
Note that identifiability requires two conditions: %, which are quite strong: 
\begin{itemize}
\item The non-degeneracy condition $\ubar \in \ri\Pa{\partial f(\xbar)}$, this ensures the minimality of the manifold $\calM$ to be identified. 

\item Subgradient convergence $\dist\Pa{\ubar, \partial f(\xk) } \to 0 $, which means there exists a sequence $\uk\in\partial f(\xk)$ with $\uk \to \ubar$. If $\seq{\xk}$ is generated by an iterative scheme, this means that each computation within the scheme should be delivered either exactly or with a vanishing error. 
\end{itemize} 
These two conditions together is quite strong in the sense that it is not clear what happens when both fail, which motivates this work.
\end{remark}

\subsection{Partly smooth operators}

Based on the above discussion, we provide the following generalization of partial smoothness from functions to operators.

\begin{definition}\label{dfn:ps-svo}% [Partly smooth set-valued monotone operators]
Let $A:\Rn \setvalued \Rn$ be a set-valued operator, then $A$ is {\it partly smooth} at $\xbar$ relative to a set $\calM \subset \Rn$ containing $\xbar$ if $\calM$ is a $C^p$-manifold around $\xbar$, $A$ is {\it upper semi-continuous} and
\begin{enumerate}[label={\rm (\roman{*})}, ref={\rm \roman{*}}]
% \item {\color{red} Upper semi-continuity here?}
\item \label{ps:regularity} 
({\bf Regularity}) for every $x$ close to $\xbar$, $A(x)$ is nonempty, closed and convex; 
\item \label{ps:sharpness} 
({\bf Sharpness}) $\normal{\calM}(\xbar) = \Lin\pa{A(\xbar)}$, $\proj_{\tangent{\calM}(x)}\pa{A(x)}$ is continuous around $\xbar$ along $\calM$;
\item \label{ps:continuity} 
({\bf Continuity}) $A$ is continuous along $\calM$ near $\xbar$.
\end{enumerate}
\end{definition}

\begin{remark}
 The definition or the idea of partly smooth operators is discussed in \cite{drusvyatskiy2014optimality,lewis2022partial}, particularly in \cite{lewis2022partial}, where a so-called ``constant rank'' property is proposed. 
 Our Definition \ref{dfn:ps-svo} offers an alternative characterization of partly smooth operators, allowing us to construct a finer local characterization to understand the identifiability, which is the main content of Section \ref{sec:identifiability}. 
\end{remark}

%\begin{remark}
The continuity condition \iref{ps:continuity} above implies that the operator $A$ is both upper and lower semi-continuous along $\calM$. 
A key implication of the sharpness condition \iref{ps:sharpness} of Definition \ref{dfn:ps-svo} is the local normal sharpness presented, as in \cite[Proposition 2.10]{LewisPartlySmooth}.

\begin{proposition}\label{prop:ps-subdifferential}
    The (limiting) subdifferential of partly smooth function is partly smooth. 
\end{proposition}

This result is a direct consequence of the properties of subdifferentials. 
The sharpness is a crucial property, as it allows us to bridge the underlying set $\calM$ and the operator. Moreover, it can be extended to the local neighbourhood of $\xbar$ along $\calM$.

\begin{proposition}[Local normal sharpness]\label{prop:lns}
Assume $A:\Rn \setvalued \Rn$ is partly smooth at $\xbar$ for $\vbar \in A(\xbar)$ along $\calM \subset \Rn$, then for all points $x \in \calM$ close to $\xbar$ satisfy
\beqn
\normal{\calM}(x) = \Lin\Pa{A(x)}  .
\eeqn
\end{proposition}
\begin{proof}% [Proof of Proposition \ref{prop:lns}]
The following is based on the proof of \cite[Proposition 2.10]{LewisPartlySmooth}.
We first show that locally along $\calM$
\beqn
\Lin\Pa{A(x)} \subset \normal{\calM}(x)  .
\eeqn
As we assume that in Definition \ref{dfn:ps-svo} that $\proj_{\tangent{\calM}(x)}\Pa{A(x)}$ is continuous around $\xbar$, denote the projection as $\eta$, then 
\beqn
\begin{aligned}
\proj_{\tangent{\calM}(x)}\Pa{A(x)} = \eta  
&\quad\Longleftrightarrow\quad
A(x) - \eta \subset \normal{\tangent{\calM}(x)}(\eta) = \pa{\tangent{\calM}(x)}^\bot = \normal{\calM}(x)  \\
&\quad\Longleftrightarrow\quad
A(x) \subset \normal{\calM}(x) + \eta  \\
&\quad\Longleftrightarrow\quad  
\Lin\pa{A(x)} \subset \normal{\calM}(x) .
\end{aligned}
\eeqn

Now assume the claim does not hold, then there exists a sequence of points $\xk \in \calM$ approaching $\xbar$ and a sequence of unit vectors $\yk \in \normal{\calM}(\xk)$ orthogonal to $\Lin\pa{A(\xk)}$. Taking a subsequence, we can suppose that $\yk \to \ybar$, since $A$ is upper semi-continuous then $\ybar \in \normal{\calM}(\xbar)$. 

Taking two arbitrary vectors $\ubar, \vbar \in A(\xbar)$, by the continuity of $A$ there exists sequences $A(\xk) \ni \uk \to \ubar$ and $A(\xk) \ni \vk \to \vbar$, and we have $\uk-\vk \in \Lin\pa{A(x)}$ and $\iprod{\yk}{\uk-\vk} = 0$.
Taking the inner product to the limit shows $\iprod{\ybar}{\ubar-\vbar} = 0$. Since $\ubar$ and $\vbar$ are arbitrary chosen, we deduce that $\ybar$ is orthogonal to $\Lin\pa{A(\xbar)}=\normal{\calM}(\xbar)$, which contradicts the fact that $\ybar$ is a unit vector in $\normal{\calM}(\xbar)$.
\end{proof}

\begin{definition}[Smooth representative]
Let $\calM \subset \Rn$ be a $C^p$-smooth manifold containing $\xbar$, $A:\Rn \setvalued \Rn$ a set-valued operator that is partly smooth at $\xbar$ relative to $\calM$. The smooth representative of $A$ around $\xbar$ is a single-valued operator $\tilde{A}: \Rn \to \Rn$ which is continuous around $\xbar$ with $\tilde{A}(x) = \proj_{\tangent{\calM}(x)}\pa{A(x)}$. %$$ \tilde{A}(x) = \proj_{\tangent{\calM}(x)}\pa{A(x)} . $$
\end{definition}

%%%%%%%%%%%
\subsection{Calculus rules}

Similar to \cite{LewisPartlySmooth}, it can be shown that under the {transversality} assumption, the set of partly smooth set-valued operators is closed under addition and pre-composition by a smooth operator. 

Consider two Euclidean spaces $\calX$ and $\calZ$, an open set $\calW\subset \calZ$ containing a point $z$, a smooth map $\Phi: \calW \to \calX$, and a set $\calM \subset \calX$. We say $\Phi$ is transversal to $\calM$ at $z$ if $\calM$ is  manifold around $\Phi(z)$, and
\beqn
\ran\pa{\nabla\Phi(z)} + \tangent{\calM}{\pa{\Phi(z)}} = \calX  ,
\eeqn
or equivalently
\beqn
\ker\pa{\nabla\Phi(z)^*} \cap \normal{\calM}{\pa{\Phi(z)}} = \ba{0}  .
\eeqn

\begin{theorem}[Composability]\label{thm:composability}
Given Euclidean spaces $\calX$ and $\calZ$, an open set $\calW\subset \calZ$ containing a point $\zbar$, a smooth map $\Phi: \calW \to \calX$, and a set $\calM \subset \calX$, suppose $\Phi$ is transversal $\calM$ at $\zbar$. If the set-valued operator $A: \calX \setvalued \calX$ is partly smooth at $\Phi(\zbar)$ relative to $\calM$, and that the composition $\nabla \Phi(\cdot)^\top A \Phi$ is partly smooth at $\zbar$ relative to $\Phi^{-1}(\calM)$.
\end{theorem}
\begin{proof}%[Proof of Theorem \ref{thm:composability}]
Due to transversality, the set $\Phi^{-1}(\calM)$ is a manifold around any point $z \in \Phi^{-1}(\calM)$ close to $\zbar$, with normal space reads
\beqn
\normal{\Phi^{-1}(\calM)}(z) = \nabla\Phi(z)^\top \normal{\calM}\pa{\Phi(z)}  ,
\eeqn
and transversality also holds at all such $z$.

\begin{enumerate}[label={\rm (\roman{*})}]
\item
Given a smooth representative $\tilde{A}$ of $A|_{\calM}$ around $\Phi(\zbar)$, it can be verified that $\Phi(\cdot)^\top \tilde{A} \Phi$ is a smooth representative of $\pa{\nabla \Phi(\cdot)^\top A\Phi}|_{\Phi^{-1}(\calM)}$ around $\zbar$, so this latter operator is continuous around $\zbar$. 
\item
By the regularity assumption of $A$ in Definition \ref{dfn:ps-svo}, 
$\Phi(\cdot)^\top A \Phi$ is regular, hence
\beq\label{eq:noempty}
\nabla\Phi(z)^\top \pa{A\Phi(z)} \neq \emptyset  .
\eeq
\item
For the normal space, we have % the normal space is parallel to the $A(\cdot)$, since
\beqn
\begin{aligned}
\Lin\Pa{(\nabla \Phi(\cdot)^\top A\Phi)(\zbar)}
&= \Lin\Pa{\nabla\Phi(\zbar)^\top {A\Phi(\zbar)}}  \\
&= \nabla\Phi(\zbar)^\top \Lin\pa{{A\Phi(\zbar)}}  \\
&\supset \nabla\Phi(\zbar)^\top \normal{\calM}\pa{\Phi(\zbar)} = \normal{\Phi^{-1}(\calM)}(z)  .
\end{aligned}
\eeqn
\item
Consider a convergent sequence of points $\zk \to \zbar$ in $\Phi^{-1}(\calM)$, and a vector $w \in \pa{\nabla \Phi(\cdot)^\top A\Phi}(\zbar)$. By \eqref{eq:noempty} there is a vector $y \in {A\Phi(\zbar)}$ such that $\nabla\Phi(\zbar)^\top y = w$. Since $\Phi(\zk) \to \Phi(\zbar)$ in $\calM$ and $A$ is continuous on $\calM$ near $\Phi(\zbar)$, there must be vectors $\yk \in {A\Phi(\zk)}$ approaching $y$. But $\Phi$ is smooth, so the vector $\nabla\Phi(\zk)^\top \yk \in \pa{\nabla \Phi(\cdot)^\top A\Phi}(\zk)$ approaches $w$ as well. \hfill $\square$
\end{enumerate}
\end{proof}

\begin{theorem}[Separability]\label{thm:separability}
For each $j=1,2,...,m$, suppose $\calH_j$ is a real Euclidean space, that the set $\calM_j \subset \calH_j$ contains the point $\xbar_j$, and the set-valued operator $A_j : \calH_j \setvalued {\calH_j}$ is partly smooth at $\xbar_j$ relative to $\calM_j$. Then the set-valued operator $\bmA:\calH_1 \times \calH_2 \times \dotsm \times \calH_m \setvalued {\calH_1} \times {\calH_2} \times \dotsm \times {\calH_m}$ defined by
\beqn
\bmA(x_1,x_2,...,x_m) = \pa{v_1,v_2,...,v_m},~~ x_j\in\calH_j,~ v_j \in A_j(x_j),~ j=1,2,...,m  ,
\eeqn
is partly smooth at $(\xbar_1,\xbar_2,...,\xbar_m)$ relative to $\calM_1\times\calM_2\times\dotsm\calM_m$.
\end{theorem}
\begin{proof}%[Proof of Theorem \ref{thm:separability}]
This easily follows from the fact that $\calM_1\times\calM_2\times\dotsm\times\calM_m$ is a manifold around $(\xbar_1, \xbar_2, \dotsm, \xbar_m)$, with 
\beqn
\begin{aligned}
\bmN_{\calM_1\times\calM_2\times\dotsm\calM_m}(x_1,x_2,...,x_m) 
&= \normal{\calM_1}(x_1) \times \normal{\calM_2}(x_2) \times \dotsm \times \normal{\calM_m}(x_m)  ,  \\
\bmA(x_1,x_2,...,x_m) &= A_1(x_1) \times A_2(x_2) \times \dotsm \times A_m(x_m)  ,
\end{aligned}
\eeqn
and $\bmA$ is regular providing that each $A_j$ is regular at $x_j,~j=1,2,...,m$ \cite[Proposition 10.5]{rockafellar1998variational}.
\end{proof}

\begin{theorem}[Sum rule]\label{thm:sumrule}
Consider set $\calM_1,\calM_2,...,\calM_m$ in a real Euclidean space $\calZ$. Suppose the set-valued operator $A_j: \calZ \setvalued \calZ$ is partly smooth at $\zbar$ relative to $\calM_j$ for each $j$. Assume further the condition
\beqn
\msum_{j=1}^{m} y_j = 0  ~~\mathrm{and}~~ y_j \in \normal{\calM_j}(\zbar) ~~\mathrm{for~each}~~ j ~~\Longrightarrow~~ y_j = 0 ~~\mathrm{for~each}~~ j.
\eeqn
Then the set-valued operator $\sum_j A_j$ is partly smooth at $\zbar$ relative $\cap_j \calM_j$.
\end{theorem}
\begin{proof}%[Proof of Theorem \ref{thm:sumrule}]
Define the following
\begin{align*}
\calX &= \calZ \times \calZ \times \dotsm \times \calZ ~~(m~\mathrm{copies}) ,  \\
\calW &= \calZ ,  \\
\Phi(z) &= (z,z,\dotsm,z) ~~\mathrm{for}~~ z\in Z ,  \\
\calM &= \calM_1 \times \calM_2 \times \dotsm \times \calM_m ,  \\
\bmA(z_1,z_2,...,z_m) &= \msum_{j} A_j z_j ~~\mathrm{for}~~ z_j \in Z,~ j=1,2,...,m.
\end{align*}
Then applying the Theorem \ref{thm:composability} and Proposition \ref{thm:separability} leads to the desired result.
\end{proof}

\begin{corollary}[Smooth perturbation]\label{coro:smooth-perturbation}
If the set-valued operator $A: \calX \setvalued \calX$ is partly smooth at the point $\xbar$ relative to the set $\calM \subset \calX$, and the operator $B: \calX \to \calX$ is smooth on an open set containing $\xbar$, then the set-valued operator $A+B$ is partly smooth at $\xbar$ relative to $\calM$.
\end{corollary}

%%%%%
\section{Identifiability}\label{sec:identifiability}

In this section, we present our main results, centering around the identifiability of the manifold -- the most valuable property of partial smoothness.
Let $A$ be partly smooth at $\xbar$ relative to $\calM$, then identifiablity implies that $\calM$ is acting as an ``attractor''. Under proper regularity and non-degenerate conditions on the dual vector, any sequence converging to the point $\xbar \in \calM$ will land on the manifold first and converges along the manifold. 
In terms of normal cone operators, identifiability is studied in \cite{drusvyatskiy2014optimality}.

The set-valued operators perspective allows us to derive a finer local geometric characterization around the point of interest, establish finite manifold identification under weaker conditions, remove non-degeneracy condition and estimate the number of steps needed for identification.

\subsection{Identifiability of partly smooth set-valued operators}

In the following we present our new identifiability result, which lays the foundation of our followup discussions.
We start with the local characterizations of partial smoothness. 
Recall the definitions of $\varepsilon$-localization (Definition \ref{def:localization}) and resolvent-regularity (Definition \ref{def:prox-reg}) for a set-valued operator.

\begin{proposition}\label{prop:local-union}
Let a set-valued operator $A:\Rn\setvalued\Rn$ be partly smooth at $\xbar$ relative to {$C^p$}-manifold $\calM \subset \Rn$. % containing \xbar\xbar for \vbar∈A(\xbar)\vbar\in A(\xbar). 
Suppose $A$ is {\it $r$-resolvent-regular} at $\xbar$ for $\ubar \in A(\xbar)$. 
For $\gamma\in]0,1/r[$ and sufficiently small $\varepsilon >0$, define the local union as
\beq\label{eq:union-U}
\localU 
\eqdef {\mathsmaller \bigcup}_{x\in \calM\cap \ball{\varepsilon}{\xbar}} \Pa{x+ \gamma \Ae(x)}  ,
\eeq 
where $\Ae$ is the $\varepsilon$-localization of $A$ arond $(\xbar,\vbar)$. Then we have
\begin{enumerate}[label={\rm (\roman*)}, ref={\rm (\roman*)}]
    \item $\Span{(\localU)}=\Rn$;
    \item $\rslt_{\gamma \Ae}(\localU)=\cM\cap \ball{\varepsilon}{\xbar}$.
\end{enumerate}
\end{proposition}

As the linear span of $\localU$ is the whole space, we call it has ``full dimension'', this is the key of showing identification. The second claim above means that locally the resolvent of $\Ae$ is well-defined. 

\begin{remark}
The key to understand Proposition \ref{prop:local-union} is the local normal sharpness (Proposition \ref{prop:lns}). When $\calM$ is polyhedral around $\xbar$, then 
\[
\localU = \Pa{ \cM\cap\ball{\varepsilon}{\xbar} } \oplus \Pa{  {\mathsmaller \bigcup}_{x\in \calM \cap \ball{\varepsilon}{\xbar}}  \gamma \Ae(x) } .
\]
Note that, as $\Ae(x) \subseteq A(x)$, 
\[
\calU_{\gamma} 
\eqdef {\mathsmaller \bigcup}_{x\in \calM\cap \ball{\varepsilon}{\xbar}} \Pa{x+ \gamma A(x)} 
\]
also has full dimension. 
\end{remark}

\begin{proof}
The key of the proof relies on the local normal sharpness. 
\begin{enumerate}[label=(\roman*), ref=(\roman*)]
    \item With $\varepsilon > 0$, define the following union of normal cone
\[
\calF \eqdef {\mathsmaller \bigcup}_{x\in \calM\cap \ball{\varepsilon}{\xbar}} \pa{x+\normal{\calM}(x)}  .
\]
Then we have $\Span{(\calF)}=\Rn$. To see this, if $\calM$ is affine or linear, the claim is quite straightforward. 

Next we show this is true for general $\calM$ by contradiction. 
Assume that it is not true, then there exists a non-zero vector $y \in \Rn \setminus \LinHull(\calF)$ such that
\beq\label{eq:iprod-y-x-v}
\iprod{y}{x + u} = 0 ,~~ \forall x \in \calM\cap \ball{\varepsilon}{\xbar},~  \forall u \in \normal{\calM}(x)  .
\eeq
Given any point $x \in \calM\cap \ball{\varepsilon}{\xbar}$ and a non-zero vector $v \in \tangent{\calM}(x)$, we have 
\[
\iprod{v}{x + u} = 0 ,~~   \forall u \in \normal{\calM}(x) .
\]
As a result, \eqref{eq:iprod-y-x-v} implies that 
\[
y \in {\mathsmaller \bigcap}_{x\in \calM\cap \ball{\varepsilon}{\xbar}} \tangent{\calM}(x) ,
\]
which apparently is not true since $y$ is non-zero and $\calM$ is curved, hence $\calF$ is of full dimension.

{\hgap}Now we prove the local union $\localU$ is of full dimension.
As above, suppose it is not true, then there exists a non-zero vector $w\in \Rn \setminus \LinHull(\localU)$ such that 
\beqn
\iprod{w}{x+\gamma u} = 0 ,~~ \forall x \in \calM\cap \ball{\varepsilon}{\xbar},~  \forall u \in \Ae(x)  ,
\eeqn
which means $w \perp \pa{x + \LinHull( \Ae(x))}$ holds for any $x \in \calM\cap \ball{\varepsilon}{\xbar}$. 
By virtue of the local normal sharpness (Proposition \ref{prop:lns}), this further implies
\beqn
w \perp \pa{x+\normal{\calM}(x)}  \Longrightarrow  w \perp \cup_{x\in \calM\cap \ball{\varepsilon}{\xbar}} \pa{x+\normal{\calM}(x)}  .
\eeqn
Since $\bigcup_{x\in \calM\cap \ball{\varepsilon}{\xbar}} \pa{x+\normal{\calM}(x)}$ is of full dimension, therefore we have $w = 0$, which contradicts with the non-zero assumption, hence proves the claim.

\item Owing to the regularity of $A$, for sufficiently small $\gamma$ and $\varepsilon$, we have
\beq\label{eq:unique2}
\rslt_{\gamma \Ae}(x+\gamma u)=x
\eeq
for $\forall x\in \cM\cap\ball{\varepsilon}{\xbar}$ and $u\in A(x)\cap\ball{\varepsilon}{\ubar}$. In addition to the continuity of $A$ relative to $\cM$, we can simply take $u\in A(x)$ for small enough $\varepsilon$. From the definition of $\localU$, \eqref{eq:unique2} leads to $J_{\gamma \Ae}(\localU)=\cM\cap \ball{\varepsilon}{\xbar}$. \hfill $\square$

\end{enumerate}
\end{proof}

We next establish the identifiability of partly smooth operators, the result leverages the properties of the local union, as outlined in Proposition~\ref{prop:local-union}. 
Let $\epsilon>0$ be small enough and denote $\scrB_{\epsilon} \subset \Rn$ be a ball at the origin with radius equals $\epsilon$, then define the following erosion of $\localU$
\[
\erosionU \eqdef \localU \ominus \scrB_{\epsilon}  .
\]
Note that $\erosionU$ is convex and also has full dimension when $\epsilon$ is small enough.

\begin{theorem}[Identifiability]\label{thm:ps-ident}
Let $A:\Rn\setvalued\Rn$ be a set-valued operator, $\calM \subset \Rn$ a {$C^p$}-manifold and $\xbar\in\calM$. Suppose the following conditions hold
\begin{enumerate}[label={\rm ({\bf A.\arabic{*}})},ref={\rm {\bf A.\arabic{*}}}, leftmargin=3.5em]
\item \label{cnd:ps}
$A$ is partly smooth at $\xbar$ relative to $\calM$;

\item \label{cnd:resol-reg}
$A$ is $r$-resolvent-regular at $\xbar$ for $\ubar \in A(\xbar)$;

\item \label{cnd:str-inc}
The strong inclusion $\ubar\in \ri\Pa{A(\xbar)}$ holds.
\end{enumerate}
Then  
\begin{enumerate}[label={\rm (\roman*)}, ref={\rm (\roman*)}]
\item Define the local union $\localU$ as in \eqref{eq:union-U} and denote $\zbar = \xbar+\gamma \ubar$, there holds
\beq\label{eq:zbar-in-localU}
\zbar \in {\rm int} \Pa{\localU}  .
\eeq

\item For any sequence $\seq{(\xk,\uk)}$ with $\uk\in\Ae(\xk)$, if 
\beq\label{eq:zk-in-erosionU}
\xk + \gamma \uk \in \erosionU ,
\eeq
then $\xk \in \calM$, and $\calM$ is called the ``active manifold'' for $\xbar$.
\end{enumerate}
\end{theorem}

\begin{remark}
Condition \iref{cnd:str-inc} is also referred to as ``non-degenerate condition'', this condition is the key to obtain the first claim. As we shall see in Proposition~\ref{prop:iden-degenerate}, this characterization also allows us to remove the strong inclusion condition, whose basic idea is that when $\ubar\in \rbd\Pa{A(\xbar)}$, we seek a (minimal) enlarged manifold such that \eqref{eq:zbar-in-localU} still holds. 
\end{remark}

\begin{remark}\label{rmk:identification} 
Condition \eqref{eq:zk-in-erosionU} indicates that 
\[
\xk + \gamma \uk \in {\rm int} \Pa{\localU} ,
\]
and the property of $\localU$ leads to the identification of $\calM$. 
\begin{itemize}
    \item Claim (ii) does not require the convergence of the sequence of either $\xk$ or $\uk$, it only requires {\it close enough}, which allows to establish identification result with errors; See Section \ref{sec:regression} for example of online sparse regression and proximal stochastic gradient descent method. 

    \item Given the sequence $(\xk,\uk)$, if for all $k$ large enough \eqref{eq:zk-in-erosionU} holds, then $\xk\in\calM$ also holds for all $k$ large enough, and this is ``finite (activity) identification'' as in \cite{hare2004identifying}. Below we also have an alternative characterization for finite identification.     
        \begin{itemize}
            \item Denote $d$ as the distance of $\zbar$ to the boundary of $\localU$, that is
                \beq\label{eq:dist}
                d = \dist\Pa{ \zbar, \bdy(\localU) } . 
                \eeq 
                Then $\ball{d}{\xbar+\gamma\ubar} \subset \localU $. 

            \item Given sequence $(\xk,\uk)$ with $\uk\in\Ae(\xk)$ such that 
                \beq\label{eq:bound}
                \lim\sup_k \dist\pa{ \xk+\gamma \uk,~ \zbar } < d. 
                \eeq
                This means that 
                \beq\label{eq:inclusion}
                \xk+\gamma\uk 
                \in \inte\Pa{\ball{d}{\xbar+\gamma\ubar}} \subset  \localU, % \bigcup_{x\in \calM\cap \ball{\varepsilon}{\xbar}} \pa{x+\gamma A(x)} ,
                \eeq
                and $\xk\in\calM$ follows. 
        \end{itemize}

    Note that condition \eqref{eq:bound} is stronger than \eqref{eq:zk-in-erosionU} is the sense that \eqref{eq:bound} implies \eqref{eq:zk-in-erosionU}, but not vice versa. 
\end{itemize}
\end{remark}

\begin{proof}
We first provide a convergence property of the dual vector, which is inspired by \cite{vaiterpartlysmooth}. For the sake of simplicity, we abuse the notation for sequence and subsequence. 

Let the sequence $\seq{\xk}$ such that 
$$
\calM\cap\ball{\varepsilon}{\xbar} \ni \xk \to \xbar
\qandq
A(\xk) \ni \uk \to \ubar .
$$
Then under condition of strong inclusion \eqref{cnd:str-inc}, there holds
\beq\label{eq:vk-in-int}
\uk \in \ri\Pa{A(\xk)}  ,
\eeq
for all $k$ large enough. We prove this calim by contradiction. 

Suppose the result is not true. This means that there at least exists a subsequence of $\uk$, such that
\beq\label{eq:vk-in-rbd}
\rbd\Pa{A(\xk)} \ni \uk \to \ubar  .
\eeq
Given any $\tilde{u}^{(k)} \in A(\xk)$, owing to the sharpness condition we have $\tilde{u}^{(k)} - \uk \in \sN_{\calM}(\xk)$. 
This means that there exists a unit normal vector $\vk \in \sN_{\calM}(\xk)$ such that
\beq\label{eq:exist-u}
\iprod{\vk}{\tilde{u}^{(k)}-\uk} \geq 0  ,~~ \forall \tilde{u}^{(k)} \in A(\xk) .
\eeq
Moreover, as $\vk$ has unit length, we can assume that, up to a subsequence, $\vk$ approaches a unit cluster point which is denoted as $\vbar$. It holds that $\vbar\in\sN_{\calM}(\xk)$ since the normal cone operator is upper semi-continuous. So does $\ubar \in \sN_{\calM}(\xbar)$.

In terms of $\tilde{u}^{(k)}$, choose it such that $\tilde{u}^{(k)} \to \tilde{u}$, then $\tilde{u} \in A(\xbar)$. 
As a result, taking \eqref{eq:exist-u} up to limit we get
\beq\label{eq:forall-v}
\iprod{\vbar}{\tilde{u}-\ubar}  \geq 0  ,
\eeq
which holds true for any $\tilde{u} \in A(\xbar)$. However, this contradicts with the fact $\ubar \in \ri\pa{A(\xbar)}$, hence \eqref{eq:vk-in-int} holds. Now we turn to prove the claims of the theorem. 
\begin{enumerate}[label={\rm (\roman{*})}, ref = {\rm \roman{*}}]

\item
{\bf Interior inclusion.} 
Under conditions \eqref{cnd:ps} and \eqref{cnd:resol-reg}, for the local union $\localU$ defined in \eqref{eq:union-U}, we know $\Span{(\localU)}=\Rn$ according to Proposition \ref{prop:local-union}~(i). Next we show that
\beqn
\zbar \in \inte\Pa{\localU} 
\eeqn
by contradiction.

{\hgap}Let $\seq{\xk}$ be a sequence converging to $\xbar$, and $A(\xk) \ni \uk \to \ubar$. Denote $\zk = \xk + \gamma \uk$, then we have $\zk\to\zbar$. 
Suppose the first claim is not true, meaning that $\zbar \in \bdy(\localU)$. 
We can choose $\xk$ and $\uk$ such that 
$$
\bdy(\localU) \ni \zk \to \zbar .
$$ 
Since $A$ is resolvent-regular at $\xbar$ for $\ubar$, when $\zk$ is close enough to $\zbar$, we have
$$
\norm{\xk-\xbar} \leq \varepsilon
\qandq
\norm{\uk-\ubar} \leq \varepsilon ,
$$
which means $\uk \in \Ae (\xk)$ belongs to the $\varepsilon$-localization of $A$. Owing to \eqref{eq:vk-in-int}, we can further choose $k$ large enough, such that
\beq\label{eq:uk-in-ri-Ae-xk}
\uk \in \ri\Pa{\Ae(\xk)}   . 
\eeq
Now since $A$ is $r$-resolvent-regular at $\xbar$ for $\vbar$, and $\gamma \in ]0, 1/r[$, we have $J_{\gamma A}(\zk)$ is well-defined, and $\xk = J_{\gamma A}(\zk) \to \xbar$ holds. 
As $\zk$ is on the boundary of $\localU$, we have
\beqn
\begin{aligned}
\zk \in \bdy\pa{\localU}  
\quad\Longrightarrow&\quad
\zk 
\in \rbd\Pa{\xk + \gamma \Ae(\xk)} 
= \xk + \gamma \rbd\Pa{ \Ae(\xk)}  \\
\quad\Longleftrightarrow&\quad
\uk = \qfrac{ \zk - \xk}{\gamma} \in \rbd\pa{\Ae(\xk)} ,
\end{aligned}
\eeqn
which contradicts with \eqref{eq:uk-in-ri-Ae-xk}. We prove the first claim. 

\item
{\bf Manifold identification.}
By Proposition \ref{prop:local-union}~(ii), we have $J_{\gamma \Ae}(\erosionU) \subseteq \cM\cap \ball{\varepsilon}{\xbar}$, so we deduce $\xk = J_{\gamma \Ae}(\xk+\gamma \uk) \in \calM$ by \eqref{eq:inclusion}, %Hence \calM\calM is identifiable at \xbar\xbar for~\ubar\ubar.  
and conclude the proof. \hfill $\square$
\end{enumerate}
\end{proof}

Before presenting examples, we need to discuss (locally) the uniqueness of the active manifold, as in \cite[Corollary 4.2]{hare2004identifying}. The strong inclusion is key to ensuring uniqueness, and as we will see later, when strong inclusion fails, uniqueness also fails.

\begin{proposition}[Uniqueness of active manifold]\label{prop:uniqueness}
Let $A:\Rn\setvalued\Rn$ be a set-valued operator, $\calM_1, \calM_2 \subset \Rn$ be {$C^p$}-manifolds that both contain $\xbar$. Assume conditions \iref{cnd:resol-reg} and \iref{cnd:str-inc} hold. 
If $A$ is partly smooth at $\xbar$ relative to both $\calM_1$ and $\calM_2$, then near $\xbar$ there holds $\calM_1 = \calM_2$. 
\end{proposition}
\begin{proof}
Suppose $\calM_1 \neq \calM_2$ locally near $\xbar$ and there exists a sequence $\seq{\xk}$ converging $\xbar$ satisfying $\xk \in \calM_1 \setminus \calM_2$. 

Since $A$ is partly smooth at $\xbar$ relative to $\calM_1$, we have that $A(\xk)$ converge to $A(\xbar)$. For the dual vector $\ubar$, there exists a sequence of dual vectors $\seq{\uk}$ with $\uk\in A(\xk)$ and $\uk\to\ubar$. Apply Theorem \ref{thm:ps-ident} for $\calM=\calM_2$, we get $\xk\in\calM_2$ which contradicts with $\xk \in \calM_1 \setminus \calM_2$. 
\end{proof}

In the following, we provide two simple examples to illustrate the geometry of local union and its role in the analysis of identification. In Section \ref{sec:examples}, example of $\ell_0$-norm is provided. 

\begin{figure}[!htb]
\centering
\includegraphics[width=0.85\linewidth, trim={5mm 5mm 5mm 2mm},clip]{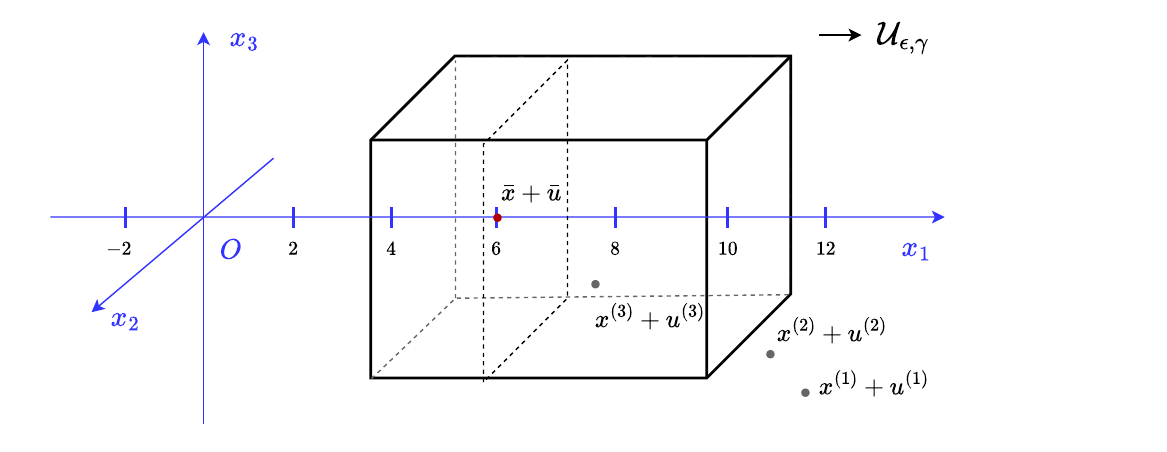}
\caption{Example of partly smooth set-valued maximally monotone operator and identifiable set.}
\label{fig:3dlasso}
\end{figure}

\begin{example}[$\ell_1$-norm]

% \todo{Add a plot here, case of 3D LASSO...}
To illustrate Theorem \ref{thm:ps-ident}, we consider the following simple example in $\RR^3$, where $A$ is a maximally monotone operator of the form, for $\forall x \in \bbR^3$
\beqn
A(x) = \sign\pa{x} + \begin{pmatrix} x_1-7 \\ x_2-1/2 \\ x_3-1/2 \end{pmatrix}  ,~\textrm{where}~~
\sign\pa{x} = 
\left\{
\begin{aligned}
&+1  ,& x > 0  ,  \\
&[-1, 1] ,& x = 0  ,  \\
&-1 ,& x < 0  .
\end{aligned}
\right.
%A(x) = \sign\pa{x} + \pa{x - b}  ,
\eeqn
%where b=(7,1/2,1/2)⊤b = (7, 1/2, 1/2)^\top.
Apparently for $\xbar = [ 6,0,0 ]^\top$ we have $\ubar = 0 \in \ri\pa{A(\xbar)}$, and also $A$ is partly smooth along the set $\calM = [\bbR,0,0]^\top$. Therefore, consider the local union $\localU$ of $x+A(x)$ of $\xbar$ along $\calM$ which is shown in Figure \ref{fig:3dlasso}, clearly, $\localU$ has full dimension.
For the sequence $\xk+\uk \to \xbar+\ubar$, the first two points  $x^{(1)}+u^{(1)}$ and $x^{(2)}+u^{(2)}$ are outside $\localU$, and starting from the 3rd point, $\xk+\uk \in \inte\pa{\localU}$ and $\uk \in \ri\pa{A(\xk)}$.

\end{example}

\begin{figure}[!htb]
\centering
 \subfloat[2D example]  
  {
      \includegraphics[width=0.47\textwidth, trim={15mm 9mm 44mm 5mm},clip]{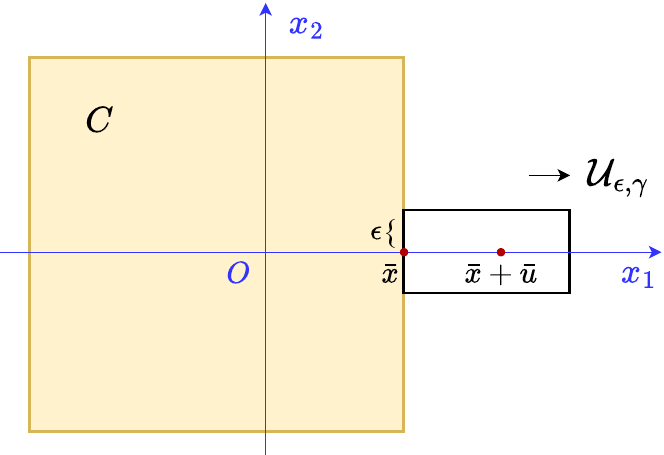}
  }
  \subfloat[3D example]
  {
      \includegraphics[width=0.47\textwidth, trim={10mm 5mm 34mm 5mm},clip]{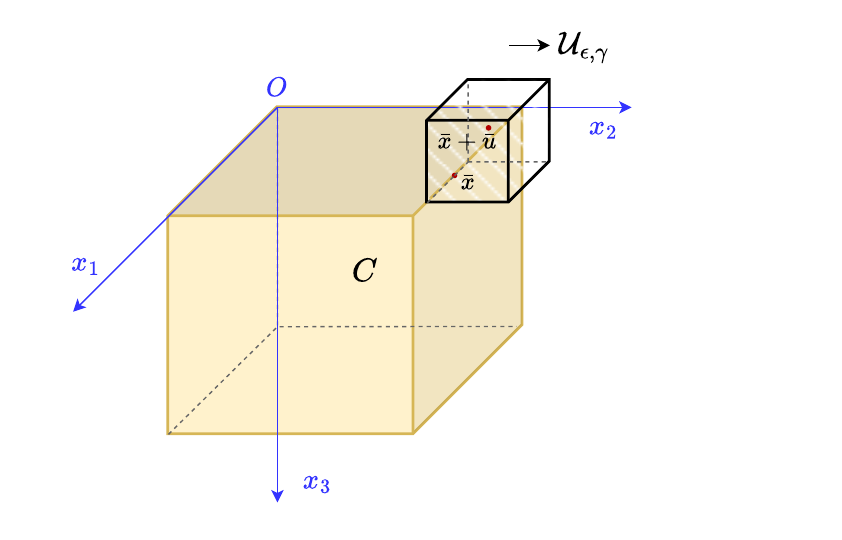}
  }
\caption{Illustrations of local unions for indicator functions. (a) A two-dimensional example. (b) A three-dimensional example.}
  \label{fig:indicator}           
\end{figure}

\begin{example}[Indicator function]
We now consider indicator function examples, in both $\RR^2$ and $\RR^3$. The indicator function is represented as
\[
\iota_C(x) = \left\{
\begin{array}{cc}
    0  &, x\in C \\
    +\infty &, x\notin C
\end{array}
\right.
\]
As shown in Figure~\ref{fig:indicator}, we present two examples of the indicator function $\iota_C$, where the set $C$ is represented by the yellow region. We then plot the local union $\localU$ based on given $\xbar$ and $\ubar$. Note that according to the definition of the local union, it suffices to consider $u$ within the $\varepsilon$-localization around $\ubar$.

\end{example}

\begin{remark}\label{rmk:dual-norm}
From the illustration of $\ell_1$-norm, it can be seen that $\localU$ is a polyhedron. This is due to the fact that the dual norm of $\ell_1$-norm is $\ell_\infty$-norm. For this case, we can refine the condition in \eqref{eq:bound} by defining the distance function
\[
\dist\Pa{\ubar, \Ae(\xk)} = \inf_{u\in\Ae(\xk)} \norm{\ubar-u}_{\infty} ,
\]
instead of the standard $\ell_2$-norm distance. This allows to provide a tighter analysis on identification, see Section \ref{sec:prox-sgd} for discussion. 
\end{remark}

For the rest of the section, we need to discuss the following problems: 1) relation with the constant rank property \cite{lewis2022partial}; 2) identifiability without strong inclusion condition \iref{cnd:str-inc}; 3) if the sequence $\seq{\xk}$ is generated by a numerical scheme, what is the upper bound for the number of steps needed for identification.

\subsection{Relation with constant rank property}

We investigate the relationship between our definition and the constant rank property of \cite{lewis2022partial}. It can be shown that under the strong inclusion condition, Definition \ref{dfn:ps-svo} naturally leads to the constant rank property, thereby establishing a connection between these two concepts.

We begin by recalling the definition of partly smooth operators via constant rank property \cite{lewis2022partial}. %, which is framed from the perspective of the constant rank.
Let $\calX, \calU$ be subsets of $\Rn$, define the canonical projection ${\rm proj} : \calX \times \calU \to \calX $ by ${\rm proj}(x,u) = x$. 

\begin{definition}[Constant rank \cite{lewis2022partial}]
\label{dfn:con-rank}
A set-valued mapping $A: \calX \rightrightarrows \calU$ is called \textit{partly smooth} at a point $\xbar\in \calX$ for $\ubar\in A(\xbar)$ when the graph $\gra(A)$ is a smooth manifold around $(\xbar,\ubar)$ and the projection ${\rm proj}$ restricted to $\gra(A)$ has constant rank around $(\xbar,\ubar)$. The dimension of $A$ at $\xbar$ for $\ubar$ is then just the dimension of its graph around $(\xbar,\ubar)$.
\end{definition}

From now on, we use the rank of the projection operator to represent the dimension of the manifold $\calM$, which is denoted as $\rank\Pa{ {\rm proj}(\gra(A),~ \calM) }$. 

\begin{proposition}\label{prop:equality} 
Under the resolvent-regularity condition \iref{cnd:resol-reg} and strong inclusion condition \iref{cnd:str-inc}, Definition \ref{dfn:ps-svo} implies Definition \ref{dfn:con-rank}. 
\end{proposition}
\begin{proof}
As $\calM$ is a manifold and $A$ is partly smooth at $\xbar$ relative to $\calM$, the graph $\gra(A)$ locally around $(\xbar, \ubar)$ is a manifold. 
From Proposition \ref{prop:uniqueness}, the resolvent-regularity of $A$ at $\xbar$ for $\ubar$ ensures locally the uniqueness of $\calM$. 
Define the localization of graph of $A$ as
$$
G_{\varepsilon} = \gra(A) ~\mathlarger{\cap}~ \Pa{ \ball{\varepsilon}{\xbar}\times\ball{\varepsilon}{\ubar} } ,
$$
we have 
\[
{\rm proj}(G_{\varepsilon}) = \calM \cap \ball{\varepsilon}{\xbar} 
\]
which is unique. Hence ${\rm proj}$ has constant rank property around $(\xbar, \ubar)$. \hfill $\square$
\end{proof}

\subsection{Strong inclusion condition fails}

The strong inclusion condition \iref{cnd:str-inc}, \aka non-degeneracy condition in optimization, is crucial to the identifiability, as it ensures the uniqueness of the active manifold, which also means the active manifold is ``minimal'' in terms of dimension. 

When \iref{cnd:str-inc} fails, so does \eqref{eq:zbar-in-localU}. 
However, the failure of strong inclusion is not always easy to characterize, take $\ell_0$-norm for example, whose subdifferential is unbounded and proximal subdifferential is an open set, see Section \ref{sec:subdiff-l0}. 
Therefore, in this part we focus on the case $A$ is moreover (locally) monotone around $\xbar$.

\begin{remark}
    In \cite{Fadili2018}, based on the ``mirror stratification'' of proper l.s.c. convex functions, the authors studied the identifiability without non-degenerate condition. It is shown that the failure of non-degeneracy results results in an enlarged manifold that contains the minimal true manifold, and identification lands on a layer between the minimal manifold and the enlarged manifold. 
    However, one drawback of their result is that it heavily relies on the dual characterization of the functions, and cannot handle the problems involving multiple partly smooth functions. 
\end{remark}

With our novel local characterization of identification, we are able to extend the identifiability to the degenerate setting. Similar to \cite{Fadili2018}, we also need an enlarged manifold, which is provided in the definition below. Recall the definition of the local union $\localU$ in \eqref{eq:union-U}, since $A$ is locally monotone around $\xbar$, $A$ is resolvent regular at $\xbar$ for any $u\in A(\xbar)$. Define the following local union parameterized by only $\gamma > 0$,
\beqn %\label{eq:union-U-gamma}
\calU_{\gamma} 
\eqdef {\mathsmaller \bigcup}_{x\in \calM\cap \ball{\varepsilon}{\xbar}} \Pa{x+ \gamma A(x)}  ,
\eeqn 
with $\varepsilon>0$ small enough such that $A$ is monotone over $\ball{\varepsilon}{\xbar}$. 

\begin{definition}[Enlarged manifold]\label{def:enlarge-manifold}
Let the set-valued operator $A:\Rn\setvalued\Rn$ be locally monotone around $\xbar$, and partly smooth at $\xbar$ relative to a {$C^p$}-manifold $\calM \subset \Rn$. 
% Suppose $A$ is $r$-resolvent-regular at $\xbar$ for $\ubar \in \rbd \Pa{ \Ae(\xbar) } $. 
% 
Let $\ubar \in \rbd\Pa{ A(\xbar) }$, then the enlarged manifold for $\ubar$ is defined by
\[
\widehat{\calM}
\in \argmin_{\rank\pa{ {\rm proj}(\gra(A),~ \calM') }} \bBa{ \calM' \subset \Rn \mid \zbar \in \inte \Pa{ \calU_{\gamma} } ,~~ \calM \subseteq \calM'  }  ,
\]
where $\zbar= \xbar+\gamma \ubar$. 
\end{definition}

Note that in the definition, the enlarged manifold requires to be the smallest manifold such that interior inclusion holds. 
We are ready to discuss the identifiability without strong inclusion condition. 

\begin{corollary}[Enlarged identification]
	\label{prop:iden-degenerate}
Suppose the set-valued operator $A:\Rn\setvalued\Rn$ is locally monotone around $\xbar$, and partly smooth at $\xbar$ relative to a {$C^p$}-manifold $\calM \subset \Rn$. Given a degenerate dual vector
\[
\ubar \in \rbd \Pa{ A(\xbar) } .
\]
There exists an enlarged manifold $\widehat{\calM}$ such that
\[
\widehat{\calU}_{\gamma}
\eqdef {\mathsmaller \bigcup}_{x\in \widehat{\calM} \cap \ball{\varepsilon}{\xbar}} \Pa{x+ \gamma A(x)}
\]
has full dimension. 
If, moreover, $A$ is {continuous} at $\xbar$ relative to $\widehat{\calM}$, then $\widehat{\calM}$ is identifiable at $\xbar$ for $\ubar$. 
\end{corollary}

The above identifiability is the direct consequence of Theorem \ref{thm:ps-ident}.

\begin{remark}
Though we can obtain identifiabilty, the identified manifold  is no longer unique. In terms of the dimension of the identified manifold, it can be euqal to that of $\calM$ or $\widehat{\calM}$, or something between. 
In Section \ref{exp:iden-degeneracy}, we provide examples to illustrate this. 
\end{remark}

\subsection{Upper bound for identificaton steps}

In either the previous result, Proposition \ref{prop:iden-function}, or our Theorem \ref{thm:ps-ident}, it is only stated that for $k$ large enough, manifold identification occurs. However, the estimation of this $k$ is not provided. 
This problem has been considered in the literature, such as \cite{liang2017local,sun2019are}. Based on new local characterization, we are able to provide a new estimation.

For the sake of simplicity, our following discussion assumes the dual vector is convergent.

\begin{proposition}\label{prop:step}
Let the set-valued operator $A:\Rn\setvalued\Rn$ be partly smooth at $\xbar$ relative to a manifold $\calM \subset \Rn$, and $r$-resolvent-regular at $\xbar$ for  
$\ubar \in \rbd \Pa{ \Ae(\xbar) } $. 
Let $K>0$ be such that
\[
\norm{(\xk+\gamma\uk)-(\xbar+\gamma\ubar)} \leq  d \eqdef \dist\Pa{ \xbar+\gamma\ubar,~\localU }
\]
holds for all $k\geq K$, then $$\xk\in\calM,~\forall k\geq K.$$
In particular, if there exists $p>0$ such that for any $k\in\bbN$, 
\beq\label{eq:ukubar-p-xkxbar}
\norm{\uk-\ubar} \leq p\norm{\xk-\xbar} .
\eeq
Then we have
\begin{enumerate}[label={\rm (\roman*)}, ref={\rm (\roman*)}]
\item If $L=\sum_{k\in\bbN}\norm{\xk-\xkm} < +\infty$, then % $K$ is such that
\[
K = \argmin_{ k } \Bba{ k \geq 0 \mid   \sum_{i=0}^{k-1} \norm{x^{(i)}-x^{(i+1)}} \geq L - \qfrac{d}{1+\gamma p} } . 
\]

\item If $\norm{\xk-\xbar} \leq C \rho^k$ for some $C> 0$ and $\rho \in ]0,1[$, then 
\[
K = \left\lceil \qfrac{ \log d - \log (1+\gamma p) -\log C }{ \log \rho } \right\rceil ,
\]
where $\lceil q \rceil,~q\in\RR$ denotes the smallest integer that is larger than $q$. 
\end{enumerate}
\end{proposition}

\begin{proof}
Note that if
\[
\begin{aligned}
\norm{(\xk+\gamma \uk) - (\xbar+\gamma\ubar)}
&\leq \norm{\xk-\xbar} + \gamma \norm{\uk-\ubar} \\
&\leq (1+\gamma p) \norm{\xk-\xbar}  \\
&\leq d ,
\end{aligned}
\]
identification happens owing to Theorem \ref{thm:ps-ident}. 
\begin{enumerate}[label={\rm (\roman*)}, ref={\rm (\roman*)}]
\item The finite sum $L=\sum_{k\in\bbN}\norm{\xk-\xkm} < +\infty$ means the sequence $\seq{\xk}$ has finite length. Therefore, given $K>0$, if we let
\[
\begin{aligned}
(1+\gamma p) \norm{x^{(K)}-\xbar}  
&\leq (1+\gamma p) \sum_{i=K}^{\infty} \norm{x^{(i)} - x^{(i+1)}} \\
&\leq (1+\gamma p) \Ppa{ L -  \msum_{i=0}^{K-1} \norm{x^{(i)} - x^{(i+1)}} } 
\leq d ,
\end{aligned}
\]
then we have
\[
\sum_{i=0}^{K-1} \norm{x^{(i)} - x^{(i+1)}}
\geq L - \qfrac{d}{1+\gamma p} .
\]
Let $K$ be the smallest integer such that the above inequality holds, then we prove the claim.

\item The second case is much more straight as $\xk$ converges linearly. As a result, if $K$ is large enough such that
\[
\begin{aligned}
(1+\gamma p) \norm{x^{(K)}-\xbar}  
\leq (1+\gamma p) \times C \rho^K
&\leq d ,
\end{aligned}
\]
then identification happens and we have
\[
K \geq \qfrac{ \log d - \log (1+\gamma p) - \log C }{ \log \rho } .
\]
Taking the smallest integer larger than the {\tt lhs} of the above inequality concludes the proof.  \hfill $\square$
\end{enumerate}

\end{proof}

\begin{remark}$~$
\begin{itemize}
\item Condition \eqref{eq:ukubar-p-xkxbar} is algorithm dependent, see Example \ref{eg:FB-step} for an illustration of Forward--Backward splitting method. 
    \item In general, in the context of non-smooth optimization or monotone inclusion, when only a sublinear convergence rate is achieved, it is typically expressed in terms of the residual $\norm{\xk-\xkm}$ rather than $\norm{\xk-\xbar}$, unless stronger assumptions are imposed \cite{liang2016convergence}. However, finite length property exists in for many problems, such as nonsmooth optimization with objective satisfying \KL inequality. Another example is the FISTA method; as shown in \cite{liang2022improving}, a modified version achieves the following convergence rate
    \[
    \sum_{k\in\bbN} k \norm{\xk-\xkp}^2 < +\infty ,
    \]
    which is analog to finite length. 

    \item Under degeneracy condition, we can also provide an upper bound estimation of identification step as stated in Proposition \ref{prop:step}. % Specifically, let $d$ be defined as in \eqref{eq:dist-larger}.
\end{itemize}
    
\end{remark}

Below we use Forward--Backward splitting method to illustrate Proposition \ref{prop:step}.

\begin{example}\label{eg:FB-step}
Consider the following monotone inclusion
\beq\label{eq:mi}
{\rm find}\quad x\in\Rn \quad {\rm such~that~}\quad  0\in A(x)+B(x) ,
\eeq
where $A:\Rn\setvalued\Rn$ is a maximally monotone and $B:\Rn\to\Rn$ is a $\beta$-cocoercive operator with $\beta>0$. Let $\{\xk\}_{k\in\bbN}$ be the sequence generated from the Forward-Backward splitting method \cite{lions1979splitting}
\[
\xkp = \Jj_{\gamma A} \pa{\Id - \gamma B}(\xk) ,~~ \gamma \in ]0, 2\beta[ .
\]
\begin{enumerate}[label={\rm (\roman*)}, ref={\rm (\roman*)}]
\item Let $\xbar$ be a solution of \eqref{eq:mi}, then 
\[
\xbar = \Jj_{\gamma A} \pa{\Id - \gamma B}(\xbar)
\]
and denote $\ubar = -B(\xbar)$.

\item 
If $B$ is strongly monotone with modulus $\kappa>0$, that is
\[
\iprod{B(x)-B(y)}{x-y} \geq \kappa \norm{x-y}^2 .
\]
Since $\Jj_{\gamma A}$ is firmly nonexpansive, we have
\[
\begin{aligned}
\norm{\xkp-\xbar}^2
&= \norm{ \Jj_{\gamma A} \pa{\Id - \gamma B}(\xk) - \Jj_{\gamma A} \pa{\Id - \gamma B}(\xbar) }^2 \\
&\leq \norm{ \pa{\Id - \gamma B}(\xk) -  \pa{\Id - \gamma B}(\xbar) }^2  \\
&\leq \norm{\xk-\xbar}^2 - 2\gamma \iprod{\xk-\xbar}{B(\xk)-B(\xbar)} + \gamma^2 \norm{ B(\xk) - B(\xbar) }^2  \\
&\leq \Pa{1-2\gamma\kappa + \sfrac{\gamma^2}{\beta^2}}  \norm{ \xk-\xbar }^2 .
\end{aligned}
\]
When $\gamma(2\kappa-\gamma/\beta^2)\in ]0,1[$, then $\rho^2 = 1-\gamma(2\kappa - \sfrac{\gamma}{\beta^2}) < 1$ and 
\[
\norm{\xk-\xsol} \leq \rho^k \norm{x^{(0)} - \xbar} .
\]
 
\item  
The update of $\xk$ yields the dual vector
\[
\uk = \qfrac{\xkm-\xk}{\gamma} - B(\xkm) \in A(\xk) .
\]
Then we have
\[
\begin{aligned}
\norm{\uk-\ubar} 
&= \norm{\sfrac{\xkm-\xk}{\gamma} -B(\xkm)+B(\xbar)} \\
&= \sfrac{1}{\gamma} \norm{\xbar-\xk + \xkm - \gamma B(\xkm)-\xbar +\gamma B(\xbar)} \\ 
&\leq \sfrac{1}{\gamma} \Pa{ \norm{\xk-\xbar} + \norm{\xkm-\xbar} } \\
&\leq \sfrac{2}{\gamma}  \norm{\xkm-\xbar} \\
&\leq \sfrac{2}{\gamma}  \norm{x^{(0)}-\xbar} \rho^{k-1} .
\end{aligned}
\]
Consequently,
\[
\begin{aligned}
\norm{\xk-\xbar}+\gamma \norm{\uk-\ubar}
&\leq \rho^k \norm{x^{(0)} - \xbar} + 2\norm{x^{(0)}-\xbar} \rho^{k-1}  \\
&\leq \rho^k \times \Pa{ \norm{x^{(0)} - \xbar} + \sfrac{2}{\rho}  \norm{x^{(0)}-\xbar}  }  . 
\end{aligned}
\]
By making $\rho^k \times \Pa{ 1+\sfrac{2}{\rho}}\norm{x^{(0)} - \xbar} < d$, we can obtain the estimation of $K$. 
\end{enumerate}
\end{example}

\begin{remark}
    We can also consider the error bound condition or strong metric regularity \cite{drusvyatskiy2013tilt} in Example~\ref{eg:FB-step}, both of which are weaker than strong monotonicity. However, these conditions can still yield similar estimates for the identification step, differing only in the associated parameter.
\end{remark}

%%%%%%%%
\section{Applications}\label{sec:applications}

In this section, we provide various examples to verify our theoretical findings. 1) We first discuss examples of partly smooth operators, including the (limiting) subdifferential of $\ell_0$-norm, monotone inclusion formulation of Primal--Dual splitting which is also discussed in \cite{lewis2022partial}, and variational inequality. 2) Then we verify the result of identification under nonconvergent dual vector via mini-batch stochastic gradient descent, and under degenerate dual vector via proximal operator of $\ell_1$-norm. 3) Lastly, we verify the upper bound of number of steps for identification via elastic net.

\subsection{Examples of partly smooth set-valued operators}
\label{sec:examples}
In this part, we provide three examples of partly smooth operators, one is the subdifferential of nonconvex function, while the other is monotone operators from Primal--Dual splitting method and variational inequality. 

\subsubsection{Subdifferential of nonconvex function}\label{sec:subdiff-l0}

As subdifferential is an important source of set-valued operators, when a function is partly smooth, its (limiting) subdifferential is also partly smooth. For convex partly smooth functions, the results in Section \ref{sec:identifiability} are quite straightforward to verify as all the regularity conditions are satisfied automatically; we refer to \cite{liang2016thesis} for convex examples. For subdifferential of nonconvex function, the local characterization is controlled by the (prox-)regularity, below we provide an example of $\ell_0$-norm.

\begin{example}[$\ell_{0}$-norm]
For $x\in\Rn$, the $\ell_{0}$-norm of $x$ is defined as
\[
f(x) = \norm{x}_0 \eqdef \#\Ba{ i:x_i\neq 0 } ,
\]
which returns the number of nonzero element in $x$. 
It can verified that $R$ is a polyhedral norm, given $\xbar \in \Rn$, denote $\calS = \ba{i : \bar{x}_i \neq 0}$, then $\ell_0$-norm is partly smooth at $\xbar$ relative to the following manifold
\[
\calM = \Ba{ x \in \RR^n \mid \supp(x) \subseteq \calS } , 
\]
which is a subspace. Therefore, the tangent space of $\calM$ at $\xbar$ is $\calM$ itself, \ie $\calT_{\calM}(\xbar) = \calM$. 
Denote $\calS^\bot = \Ba{1,...,n} \setminus \calS$ the complement of the $\calS$, the subdifferential at $\xbar$ then is
\[
\partial f(\xbar) = \Span\Ba{e_i,~ i\in\calS^\bot} ,
\]
where $e_i$ is the $i$'th standard normal basis of $\Rn$. 
\end{example}

For the sake of simplicity, for the discussion below we let $n=2$ and $\xbar = \begin{bmatrix} 1 & 0 \end{bmatrix}^\top$. For this setting, we have
\[
\calM = \calT_{\calM}(\xbar) = \begin{bmatrix} \RR \\ 0 \end{bmatrix} 
\qandq
\partial f(\xbar) = \begin{bmatrix} 0 \\ \RR \end{bmatrix}  .
\]
Given $\varepsilon > 0$ and $\gamma>0$, define
\[
\begin{aligned}
{\calU}
&= {\mathsmaller \bigcup}_{x\in \calM\cap \ball{\varepsilon}{\xbar}} \Pa{x + \gamma \partial f(x)}  
= \begin{bmatrix} [1-\varepsilon, 1+\varepsilon] \\ \RR \end{bmatrix} .
\end{aligned}
\]
It is clear that $\Lin({\calU}) = \RR^2$. However, $\prox_{\gamma f}(\calU) \neq \calM\cap \ball{\varepsilon}{\xbar}$ as $\calU$ is unbounded and regularity is not considered. 
Let $\ubar = \begin{bmatrix} 0 & 0 \end{bmatrix}^\top \in\partial f(\xbar)$, then it can be verified that $f$ is $r$-prox-regular at $\xbar$ for $\ubar$ for some $r>0$. Denote $(\partial f)_{\varepsilon}$ the localization of $\partial f$, let $\gamma \in ]0, 1/r[$ and define
\[
\begin{aligned}
{\calU_{\varepsilon,\gamma}}
&= {\mathsmaller \bigcup}_{x\in \calM\cap \ball{\varepsilon}{\xbar}} \Pa{x + \gamma (\partial f)_{\varepsilon}(x)}  
= \begin{bmatrix} [1-\varepsilon, 1+\varepsilon] \\ [-\gamma\varepsilon, \gamma\varepsilon] \end{bmatrix} .
\end{aligned}
\]
Given that $\gamma\varepsilon \leq \sqrt{2\gamma} < 1-\varepsilon$, we have
\[
\prox_{\gamma f}(\calU_{\varepsilon, \gamma}) = \calM\cap \ball{\varepsilon}{\xbar} . 
\]
Suppose now we have a sequence $\seq{\xk}$ with $\xk\to\xbar$ and $\partial f(\xk) \ni \uk \to \ubar$, then necessarily for all $k$ large enough, we have
\[
\uk \in (\partial f)_{\varepsilon}(\xk)
\qandq
\xk + \gamma \uk \in \calU_{\varepsilon, \gamma} .
\]
As a result $\xk \in \calM\cap\ball{\varepsilon}{\xbar}$.

\subsubsection{Primal-Dual splitting}

The second example we consider is the Primal-Dual splitting method \cite{chambolle2011first,combettes2012Primal}, the convergence property of the method under partial smoothness is studied in \cite{Liang2018}, later on in \cite{lewis2022partial} it is used an example of partly smooth operators. For the sake of completeness, we provide the derivation here as an example of partly smooth operator which is not the subdifferential of partly smooth functions. 

Consider the following saddle point problem
\begin{equation}\label{eq:saddle}
	\min_{x\in \Rn}\max_{y\in \Rm}~ f(x)+ p(x) + \langle Kx,y\rangle - g^*(y)-q^*(y)
\end{equation}
where 
\begin{itemize}
	\item $K:\Rn\to \Rm$ is a bounded linear operator.
	\item $f: \Rn\to\bar{\RR}$ and $g: \Rm\to\bar{\RR}$ are proper l.s.c. and convex functions.
	\item $p:\Rn\to\bbR$ is convex differentiable with $\beta_p$-Lipschitz continuous gradient.
	\item $q^*:\Rm\to\bbR$ is convex differentiable with $\beta_{q^*}$-Lipschitz continuous gradient
\end{itemize}
Given a saddle point $(\xbar, \ybar)$, the associated optimality condition reads
\[
\begin{aligned}
	0 &\in \partial f(\xbar) + \nabla p(\xbar) + K^\top \ybar , \\
	0 &\in - K\xbar + \partial g^*(\ybar) + \nabla q^*(\ybar) , 
\end{aligned}
\]
which can be written as 
\beq\label{eq:optimal}
0\in  \underbrace{\begin{bmatrix}
		\partial f & K^\top \\
		-K & \partial g^*
\end{bmatrix}}_A \begin{bmatrix}\xbar\\ \ybar\end{bmatrix}  +
\underbrace{\begin{bmatrix}
		\nabla p & 0 \\ 0 & \nabla q^*
\end{bmatrix}}_B \begin{bmatrix}\xbar \\ \ybar\end{bmatrix} . 
\eeq
Where $A$ is maximally monotone and $B$ is co-coercive. 
By denoting $z=\begin{bmatrix}x \\ y\end{bmatrix}$, the original problem \eqref{eq:saddle} is equivalent to solve the following monotone inclusion problem
\begin{equation}\label{eq:GE}
	{\rm find}\quad z\in\Rn\times\Rm  \quad {\rm such~that}\quad 0 \in A(z)+B(z) . 
\end{equation}

Note that $A$ has the decomposition
\[
A = 
\begin{bmatrix}
		\partial f & 0 \\
		0 & \partial g^*
\end{bmatrix}
+
\begin{bmatrix}
		0 & K^\top \\
		-K & 0
\end{bmatrix} . 
\]
Hence according Proposition \ref{prop:ps-subdifferential} and Corollary \ref{coro:smooth-perturbation}, we have the following result regarding the partial smoothness of $A$.

\begin{proposition}
For the optimization problem \eqref{eq:saddle}, given a saddle point $(\xbar, \ybar)$ suppose the following holds
\begin{itemize}
\item $f$ is partly smooth at $\xbar$ relative to $\calM_{\xbar}$;
\item $g^*$ is partly smooth at $\ybar$ relative $\calM_{\ybar}$;
\end{itemize} 
Then the set-valued operator $A$ in \eqref{eq:optimal} is partly smooth. 
\end{proposition}

\subsubsection{Variational inequalities}

Examples of partly smooth operators also arise from variational inequalities. 
First, consider the following optimization problem with linear constraint
\beq\label{eq:f_plus_g}
\begin{aligned}
\min_{x\in\Rn,~ y\in\Rm}~~& f(x) + g(y) \\
{\rm subject~to}~~& Cx + Dy = e ,
\end{aligned}
\eeq
where the following assumptions are imposed
\begin{itemize}
	\item $f: \Rn\to\bar{\RR}$ and $g: \Rm\to\bar{\RR}$ are proper l.s.c. and convex functions.
	\item $C:\Rn\to \RR^\ell,~ D: \Rm\to \RR^\ell $ are bounded linear operators.
\end{itemize}
We also assume that the problem is well-posed such that solution exists. Let $\lambda$ be the dual vector, then the Lagrangian associated to \eqref{eq:f_plus_g} reads
\[
\calL(x,y, \lambda) = f(x) + g(y) - \iprod{\lambda}{Cx+Dy-e}  .
\]
Let $(\xbar,\ybar)$ and $\bar{\lambda}$ be primal and dual optimal, then there holds
\beq\label{eq:opt_x_y_lam}
\begin{aligned}
\iprod{x-\xbar}{ \tilde{\nabla } f(\xbar) - C^\top \bar{\lambda} } &\geq 0,~~~ \forall x\in\Rn \qandq \tilde{\nabla } f(\xbar) \in \partial f(\xbar) , \\
\iprod{y-\ybar}{ \tilde{\nabla } g(\ybar) - D^\top \bar{\lambda} } &\geq 0,~~~ \forall y\in\Rm \qandq \tilde{\nabla } g(\ybar) \in \partial g(\xbar) , \\
\iprod{\lambda-\bar{\lambda}}{ C\xbar+D\ybar-e } &\geq 0,~~~ \forall \lambda\in\RR^\ell . 
\end{aligned}
\eeq
Denote $\cZ=\Rn\times \Rm\times \Rm$, 
\beqn%\label{eq:vi-para}
z = \begin{bmatrix}
    x\\y\\ \lambda
\end{bmatrix}
\qandq
A(z) =
\begin{bmatrix}
    \partial f(x) -C^\top\lambda\\
    \partial g(y) -D^\top\lambda \\ 
    Cx+Dy-e 
\end{bmatrix} .
\eeqn
Note that $A$ is maximally monotone. 
Then solving the optimization problem is equivalent to solving the follow variational inequality
\beq\label{eq:VI}
\iprod{z-\zbar}{\wbar}\geq 0,\ ~\textrm{ for }~ \forall z\in \Omega \qandq \wbar \in A(\zbar) . 
\eeq
Similar to the case of Primal-Dual splitting method, under partial smoothness assumptions of $f$ and $g$, the set-valued operator $A$ is partly smooth.

\subsection{Support identification of online regression}\label{sec:regression}
Consider the following online regression problem:
\beq\label{eq:online}\tag{$P_{\bbE}$}
\min_{x\in\Rn} \mu \norm{x}_1 + \bbE_{(a,b)\sim \calD} \sfrac{1}{2}(a^\top x - b)^2,
\eeq
where $\calD$ is a Gaussian distribution with: let $\sigma_1,\sigma_2 > 0$
\[
a \sim \calN(0, \sigma_1^2)
\qandq
b = a^\top \tilde{x} + \epsilon,~~ \epsilon \sim \calN(0, \sigma_2^2)
\]
where $\tilde{x}$ is a $\kappa$-sparse vector ($\kappa \ll n$).  

As problem \eqref{eq:online} involves expectation, which makes the problem not easy to solve. A popular approach is to approximate the problem through sampling, which results to the following finite sum problem
\beq\label{eq:lasso}\tag{$P_{m}$}
\min_{x \in \bbR^n} \mu \norm{x}_1 + \frac{1}{m} \sum_{i=1}^{m} \sfrac{1}{2}(a_i^\top x - b_i)^2,
\eeq
where $(a_i,b_i)$ are samples from $\calD$. The expectation is approximated by the mean of the samples, and the resulting problem is the classic LASSO problem \cite{tibshirani1996regression}. This method is also referred to as the ``sample average approximation'' (SAA) method in statistics, see \cite{shapiro2021lectures,kleywegt2002sample} for more details.

The properties of problem \eqref{eq:lasso} have been widely studied in the literature, including algorithm, convergence rate, stability and dimension reduction, etc. For instance, ISTA \cite{daubechies2004iterative} and FISTA \cite{beck2009fast} are two popular algorithms to solve the problem, and both of them achieve local linear convergence \cite{liang2017activity}, the stability property of the solution is studied in \cite{bickel2009simultaneous}. 
In comparison, the choices for the online problem \eqref{eq:online} are rather limited. For instance typical algorithm to solve the problem would be proximal stochastic gradient descent (Prox-SGD) \cite{robbins1951stochastic}. 
A mini-batch version of the Prox-SGD is described below: let $s$ be the batch size
\beq\label{eq:psgd}
\left\lfloor
\begin{aligned}
&\textrm{Sample the index subset $\calI_k \subset \Ba{1,2,...,m}$ with size $s$} ,  \\[1mm]
&~ \xkp = \prox_{\mu\alpha_k \norm{\cdot}_1} \bPa{ \xk - \qfrac{\alpha_k}{s} \msum_{i\in\calI_k} \Pa{a_i^\top\xk - b_i} a_i } , 
\end{aligned}
\right.
\eeq
where $\alpha_k$ is step-size. 
It is well-known that stochastic gradient has non-vanishing error (or bounded variance), %see below for a simple derivation, 
this is the reason why Prox-SGD does not have identification property as pointed out in \cite{poon2018proxsagasvrg}. However, when mini-batch is considered, with Theorem \ref{thm:ps-ident} the new condition for identification, we can show that identification occurs when batch size is large enough.

We denote $\xsol, \xsol_m$ the optimal solutions of \eqref{eq:online} and \eqref{eq:lasso}, respectively. We also assume that they are unique. According to \cite[Theorem 5.7]{shapiro2021lectures}, under proper assumptions one has 
\[
\bbE[\norm{\xsol_m-\xsol}] = O(1/\sqrt{m})
\]
As a result, in the following we discuss two problems
\begin{itemize}
    \item For Prox-SGD, under what choices of batch size $s$, $\xk$ can identify the support of $\xsol_m$.
    \item Between \eqref{eq:online} and \eqref{eq:lasso}, under how many samples, $\xsol_m$ can identify the support of $\xsol$. 
\end{itemize}
Given that both $\xsol, \xsol_m$ are non-degenerate with respect to their probelms.

\subsubsection{Support identification of $\xk$ with respect to $\xsol_m$}\label{sec:prox-sgd}

In this part, we assume that the sequence $\xk$ is convergent to $\xsol_m$.

%%%%%%%%%%%%%%%%%%%%%%%%%%%%%%%%%%%%%%%%%%%%%%%%%%%%%%%%%%
\paragraph{Local union} Since we are in the convex setting, $\partial \norm{\cdot}_1$ is maximal monotone. For some small enough $\varepsilon > 0$, the local union
\[
\calU_m \eqdef {\mathsmaller \bigcup}_{x\in \calM\cap \ball{\varepsilon}{\xsol_m}} \Pa{x+ \partial \norm{x}_1 }  ,
\]
has full dimension. 
The dual vector is denoted as
\[
\usol_m \eqdef - \qfrac{1}{m\mu} \sum\nolimits_{i=1}^{m} \Pa{a_i^\top\xsol_m - b_i} a_i \in \ri\Pa{ \partial \norm{\xsol_m}_1 }  .
\]
which is assumed to be non-degenerate. 
As a consequence, the distance
\[
d = \dist\Pa{ \xsol_m+\usol_m, \bdy(\calU_m) } 
\]
is strictly positive.

%%%%%%%%%%%%%%%%%%%%%%%%%%%%%%%%%%%%%%%%%%%%%%%%%%%%%%%%%%
% \vskip2mm
\paragraph{Dual vector} 
From the definition of proximal operator and the iteration \eqref{eq:psgd}, we have
\[
\ukp \eqdef \qfrac{\xk-\xkp}{\mu\alpha_k} - \qfrac{1}{s\mu} \sum\nolimits_{i\in\calI_k} \Pa{a_i^\top\xk - b_i} a_i \in \partial \norm{\xkp}_1 . 
\]
Denote $\calI_k^\bot = \Ba{1,2,...,m} \setminus \calI_k$, then we have
\beq\label{eq:error_sgd}
\begin{aligned}
&\norm{(\xkp+\ukp) - (\xsol_m+\usol_m)} \\
&\leq  \norm{\xkp - \xsol_m} + \norm{\ukp - \usol_m} \\
&\leq  \norm{\xkp - \xsol_m} + \norm{\sfrac{\xk-\xkp}{\mu\alpha_k} - \sfrac{1}{s\mu} \msum_{i\in\calI_k} \pa{a_i^\top\xk - b_i} a_i + \sfrac{1}{m\mu} \msum_{i=1}^{m} \pa{a_i^\top\xsol_m - b_i} a_i } \\
&\leq  \norm{\xkp - \xsol_m} + \sfrac{1}{\mu\alpha_k}\norm{\xk-\xkp} + \sfrac{1}{\mu} \norm{\sfrac{1}{s} \msum_{i\in\calI_k} a_i^\top \pa{\xk - \xsol_m} a_i + \sfrac{1}{s} \msum_{i\in\calI_k} \pa{a_i^\top\xsol_m-b_i} a_i  \\
&\quad - \sfrac{1}{m} \msum_{i\in\calI_k} \pa{a_i^\top\xsol_m - b_i} a_i - \sfrac{1}{m} \msum_{i\in\calI_k^\bot} \pa{a_i^\top\xsol_m - b_i} a_i} \\
&\leq  \norm{\xkp - \xsol_m} + \sfrac{1}{\mu s} \norm{ \msum_{i\in\calI_k}  a_i^\top\pa{\xk - \xsol_m} a_i } + \sfrac{1}{\mu\alpha_k}\norm{\xk-\xkp} \\
&\quad + \sfrac{m-s}{\mu ms} \norm{\msum_{i\in\calI_k}\pa{a_i^\top\xsol_m-b_i} a_i} + \sfrac{1}{\mu m} \norm{ \msum_{i\in\calI_k^\bot}  \pa{a_i^\top\xsol_m - b_i} a_i }  \\
&\leq \Pa{ 1+ \sfrac{\alpha^2}{\mu} } C_k + \sfrac{1}{\mu\alpha_k}\norm{\xk-\xkp} + \sfrac{m-s}{\mu ms} \norm{\msum_{i\in\calI_k}\pa{a_i^\top\xsol_m-b_i} a_i} + \sfrac{1}{\mu m} \norm{ \msum_{i\in\calI_k^\bot}  \pa{a_i^\top\xsol_m - b_i} a_i }.
\end{aligned}
\eeq
where $\alpha = \max_i \norm{a_i}$, and $C_k = \max\Ba{ \norm{\xkp-\xsol_m}, \norm{\xk-\xsol_m} } \to 0$. 
\begin{itemize}
    \item When $\calI_k \equiv \Ba{1,2,...,m}$, \ie the full-batch proximal gradient descent, we can afford constant non-vanishing step-size $\alpha_k \geq \alpha > 0$ and consequently $\sfrac{1}{\alpha_k}\norm{\xk-\xkp} \to 0$. We assume that $\sfrac{1}{\alpha_k}\norm{\xk-\xkp} \to 0$ still holds when batch size $s$ is large enough. 

    \item For the last two terms of the line of \eqref{eq:error_sgd}, they both vanish when $s=m$, meaning that their values are small when $s$ is close to $m$.

    \item Due to the assumption of the convergence of $\xk$ to $\xsol_m$, then $k$ is large enough we also have $C$ is small enough.
\end{itemize}
Summarize from the above discussions, when $k$ and $s$ are large enough, the following condition
\beq\label{eq:leq_d}
\norm{(\xkp+\ukp) - (\xsol_m+\usol_m)} < d 
\eeq
could hold. As a result, support identification occurs.

\begin{figure}[!hb]
    \centering
    \subfloat[Identification v.s. batch size]{\includegraphics[width=0.45\textwidth]{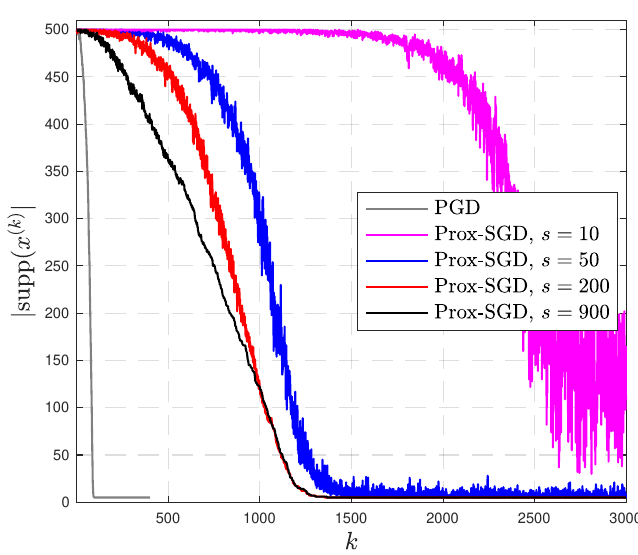}\label{fig:SL_objectives_200}}
    \hskip7mm
    \subfloat[The green dashed line is the estimation of $\dist(\xsol_m+\vsol_m,~\cU_m)$]{\includegraphics[width=0.45\textwidth]{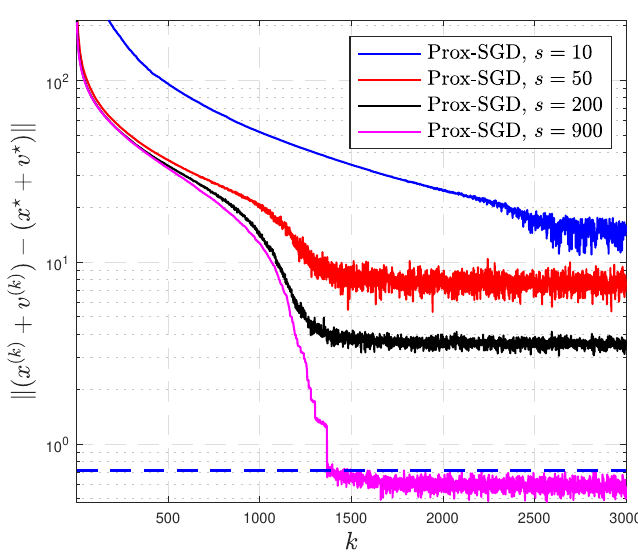}\label{fig:SL_objectives_20}}
    \\
    \caption{Identification of proximal mini-batch stochastic gradient, and dual vector error.}    
    \label{fig:minibatch_SGD_and_dual}
\end{figure}

To illustrate our result, we conduct experiment with a problem size of $(m, n) = (1000, 500)$. The identification results for different batch sizes are presented in Figure \ref{fig:minibatch_SGD_and_dual}~(a), proximal gradient descent is added for reference. We can observe from the figure that
\begin{itemize}
\item When the batch size is small, \eg $s = 10$, no identification happens. Increase the batch size can significantly reduce the support size of $\xk$.

\item For batch size equals $150, 900$, we have identification. 
\end{itemize}
We also provide the plot for error $\norm{(\xk+\uk) - (\xsol_m+\usol_m)}$, as showed in Figure \ref{fig:minibatch_SGD_and_dual}~(b):
\begin{itemize}
\item The {\it blue dashed} line is the estimation of $d=\dist{(\xsol_m+\usol_m,~\cU)}$.
\item Note that even for $s=900$ the magenta line, eventually the error is not always below $d$. Then according to our theory, identification should not occur. No to mention $s=150$, which contradicts with identification in Figure \ref{fig:minibatch_SGD_and_dual}~(a). 
\end{itemize}
We remark that this is not a contradiction, and is due to the reason that our analysis is not tight as pointed our in Remark \ref{rmk:dual-norm}. More precisely, this is because we are not fully exploiting the local geometry of $\ell_1$-norm, as we should use $\ell_\infty$-norm to characterize the error, \ie $\norm{(\xk+\uk) - (\xsol_m+\usol_m)}_{\infty}$, as $\ell_\infty$-norm is the dual norm of $\ell_1$-norm. While in our analysis above, we use $\ell_2$-norm, which makes the estimation of the error rather weak. 
In Figure \ref{fig:minibatch_SGD_infty} below, we provide the error in $\ell_\infty$-norm, which now matches with Figure \ref{fig:minibatch_SGD_and_dual}~(a).

\begin{figure}[!htb]
    \centering
    \includegraphics[width=8cm]{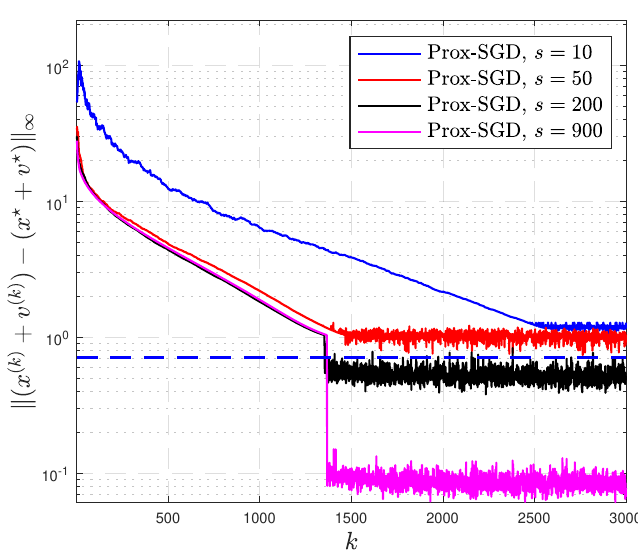}
    \caption{Error via $\ell_\infty$-norm.}
    \label{fig:minibatch_SGD_infty}
\end{figure}

\subsubsection{Support identification of $\xsol_m$ with respect to $\xsol$}
We now show that $\xsol_m$ identifies the support of $\xsol$ with high probability when $m$ is sufficiently large. 
\[
\calU \eqdef {\mathsmaller \bigcup}_{x\in \calM\cap \ball{\varepsilon}{\xsol}} \Pa{x+ \partial \norm{x}_1 }  ,
\]
has full dimension. 
The dual vector is denoted as
\[
\usol \eqdef - \bbE_{(a,b)\sim\calD} \Pa{a^\top\xsol - b} a \in \ri\Pa{ \partial \norm{\xsol}_1 }  .
\]
which is assumed to be non-degenerate. The distance is then
\[
d^\star = \dist\Pa{ \xsol+\usol, \bdy(\calU) } 
\]
is strictly positive. 
If we can have  
\[
\bbE[\dist{(\xsol_m+\usol_m, \xsol+\usol)}] < \dsol
\]
then $\xsol_m$ identifies the support of $\xsol$. Note that we have
\[
\begin{aligned}
\norm{\usol_m-\usol}
&= \norm{\sfrac{1}{\mu m}\msum_{i=1}^m(a_i^\top \xsol_m-b_i)a_i - \sfrac{1}{\mu}\bbE_{(a,b)\sim\calD} (a^\top \xsol-b) a} \\
&= \sfrac{1}{\mu}\norm{\sfrac{1}{ m}\msum_{i=1}^m(a_i^\top \xsol_m-b_i)a_i -\bbE_{(a,b)\sim\calD}(a^\top \xsol-b)a} \\
&\leq \sfrac{1}{\mu} 
 \norm{\underbrace{ \sfrac{1}{m} \msum_{i=1}^m a_i a_i^\top(\xsol_m - \xsol) }_{A}} + \sfrac{1}{\mu} \norm{\underbrace{ \sfrac{1}{m} \msum_{i=1}^m (a_i^\top \xsol - b_i) a_i - \bbE[(a^\top \xsol - b) a] }_{B}}
\end{aligned}
\]
\begin{itemize}
\item 
For term $A$, we have
\[
\bbE \left[\norm{\sfrac{1}{m} \msum_{i=1}^m a_i a_i^\top (\xsol_m - \xsol)} \right]\leq \bbE \left[ \norm{ \sfrac{1}{m} \msum_{i=1}^m a_i a_i^\top} \cdot \norm{\xsol_m - \xsol} \right]
\]
where $\bbE \left[\norm{\sfrac{1}{m} \msum_{i=1}^m a_i a_i^\top }\right]$ is bounded.
\item
In terms of $B$, we have
\[
\bbE \left[ \sfrac{1}{m} \msum_{i=1}^m (a_i^\top \xsol - b_i) a_i - \bbE[(a^\top \xsol - b) a] \right] = O\left( \frac{1}{\sqrt{m}} \right ) 
\]
by the central limit theorem.
\end{itemize}
To sum up, the total expected distance becomes:
\[
\begin{aligned}
\mathbb{E} \left[\norm{(\xsol_m+\usol_m)-(\xsol+\usol)} \right] &\leq \mathbb{E} \left[ \| \xsol_m - \xsol \| + \sfrac{1}{\mu} ( \norm{A }  + \norm{B} ) \right ] \\
&\leq O\left( {1}/{\sqrt{m}} \right ) + \sfrac{1}{\mu} \left( O\left( {1}/{\sqrt{m}} \right ) + O\left( {1}/{\sqrt{m}} \right ) \right ) \\
&= O\left( {1}/{\sqrt{m}} \right )
\end{aligned}
\]
When $m$ is sufficiently large, this distance becomes smaller than the safety distance $\dsol$, ensuring that $\xsol_m$ and $\xsol$ share the same support under the non-degeneracy condition.

\subsection{Identification properties and iteration bounds}

In this part, we turn to the identification property of partial smoothness and discuss its behavior under much weaker conditions compared to the existing result. 
We first discuss the identification result under the setting that the dual vector is not convergent using proximal mini-batch stochastic gradient, then we discuss the scenario where the non-degeneracy condition fails. We conclude this part by check the upper bound on the number of iterations need for identification. 

\subsubsection{Identification under degeneracy}\label{exp:iden-degeneracy}

In this part, we present a toy example to explain the identification property under degeneracy condition.
Consider the following simple 
\[
\min_x~ \lambda f(x) + \qfrac{1}{2}\norm{x-b}^2
\]
where $f(x)$ takes $\ell_1$-norm and nuclear norm. 
Note that solution of the problem is rather straightforward to obtain as it is computing the proximal operator of $f$. However, we can solve it with Forward--Backward splitting method \cite{lions1979splitting} with small step-size to observe the degenerate behavior. 

Below we discuss case by case by fixing $\lambda=1$, since the examples are only for illustrative purpose, the problems are rather small. 
The MATLAB variable-precision arithmetic {\tt vpa} is used such that we can use high precision computation, and we use $\norm{\xk-\xbar} \leq 10^{-20}$ as stopping criterion. 

\paragraph{$\ell_1$-norm $f(x) = \norm{x}_1$}

Let $b = \big[ 3 , 1 , 0.5 \big]^\top$, then the solution is $\xbar = \big[ 2 , 0 , 0 \big]^\top$. Moreover, we have the corresponding dual vector $\ubar$ reads
\[
\ubar = b - \xbar = \big[ 1 , 1 , 0 \big]^\top \in \bdy \Pa{ \partial\norm{\cdot}_1(\xbar) } .
\]
which violates the non-degeneracy condition at the second element as $\xbar_2=0$ while $\ubar_2=1$. 
Therefore, the minimal manifold and the enlarged manifold are
\[
\cM= \big[ \bbR,0,0 \big]^\top 
\qandq
\widehat\cM=\big[ \bbR,\bbR,0 \big]^\top .
\] 
respectively. 

Two choices of starting point $x^{(0)}$ are considered
\[
\begin{aligned}
&\textrm{Choice 1: $x^{(0)} = \big[ 2, 2 , 2 \big]^\top$} 
\qqandqq
\textrm{Choice 2: $x^{(0)} = \big[ 0, -2 , -2 \big]^\top$} .
\end{aligned}
\]
For both choices, the support size of $\seq{\xk}$ are provided in Figure~\ref{fig:degenerate}~(a)
\begin{itemize}
\item For Choice 1 starting point, the sequence $\seq{\xk}$ identifies the enlarged manifold. 

\item For Choice 2 starting point, the sequence $\seq{\xk}$ identifies the correct minimal manifold of $\xbar$. 
\end{itemize}
The above difference indicates that the when the problem is degenerate, the manifold identified depends on the direction of $x^{(0)}$ relative to the solution $\xbar$.

\begin{figure}[htbp]
\centering  
\subfloat[Support size of $\xk$]{   
\begin{minipage}{0.49\linewidth}\label{fig:support}
\centering    
\includegraphics[width=0.8\linewidth]{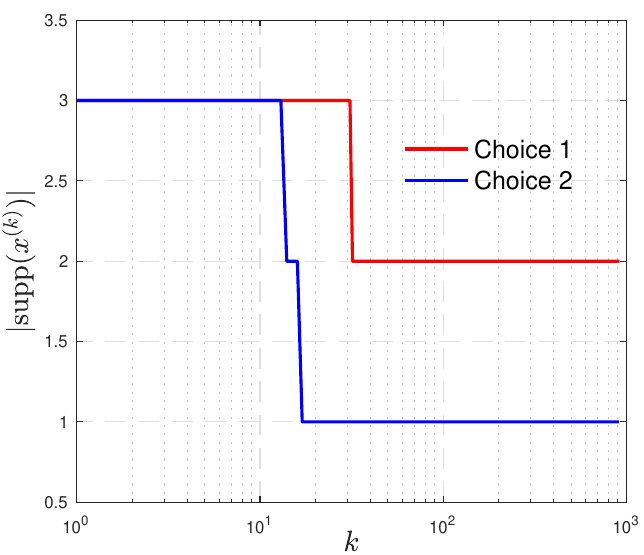}  
\end{minipage}
}
\subfloat[Rank size of $\xk$]{ 
\begin{minipage}{0.49\linewidth}\label{fig:track}
\centering    
\includegraphics[width=0.8\linewidth]{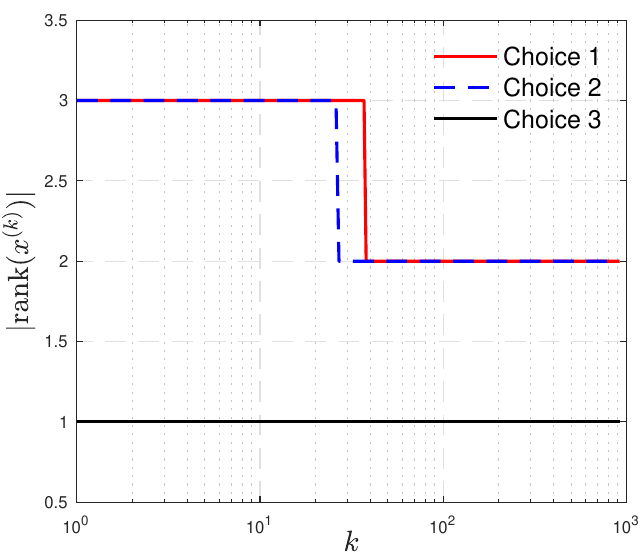}
\end{minipage}
}
\caption{Manifold identification of $\ell_1$-norm and nuclear norm under degenerate condition.}
\label{fig:degenerate}   
\end{figure}

\paragraph{Nuclear norm $f(x) = \norm{x}_*$}

Let $b\in\RR^{3\times 3}$ with singular value decomposition (SVD) as $[u,s,v]={\tt svd}(b)$ with $s = \diag\Pa{[ 3 , 1 , 0.5 ]}$. With $\lambda=1$, the singular value of $\xbar$ is $\big[ 2 , 0 , 0 \big]$. Correspondingly, the dual vector $\ubar$ reads
\[
\ubar = b - \xbar = u \ \diag\Pa{ [ 1 , 1 , 0 \big] } \ v^\top \in \bdy \Pa{ \partial\norm{\cdot}_*(\xbar) } .
\]
which violates the non-degeneracy condition at the second singular value. 
Therefore, the minimal manifold and the enlarged manifold are
\[
\cM= \Ba{ x \in\RR^{3\times 3} \mid \rank(x) = 1 }
\qandq
\widehat\cM= \Ba{ x \in\RR^{3\times 3} \mid \rank(x) = 2 }  .
\] 
respectively. 

Two choices of starting point $x^{(0)}$ are considered
\[
\begin{aligned}
\textrm{Choice 1: $x^{(0)} = u \ \diag\Pa{ [ 3 , 3 , 3 \big] } \ v^\top$} , \quad
\textrm{Choice 2: $x^{(0)} = {\tt randn}(3,3)$} 
~\qandq~
\textrm{Choice 3: $x^{(0)} = 0$}  .
\end{aligned}
\]
For all choices, the rank of $\seq{\xk}$ are provided in Figure~\ref{fig:degenerate}~(b), in this experiment, we treat values smaller than $10^{-30}$ as zero
\begin{itemize}
\item For Choice 1 \& 2 starting points, the sequence $\seq{\xk}$ identifies the enlarged manifold. 

\item For Choice 3 starting point, the sequence $\seq{\xk}$ identifies the correct minimal manifold of $\xbar$. 
\end{itemize}
Compared to $\ell_1$-norm case, since the singular values are non-negative and the direction is encoded in the eigenvectors, we found that $\xk$ always identifies the enlarged manifold as long as $x^{(0)} \neq 0$.

\subsubsection{Upper bound of number of steps for identification}
To conclude this section, we illustrate the estimation of the identification step with the following optimization problem
\beq\label{eq:penalty}
\min_x \qfrac{1}{2}\norm{Ax-b}_2^2 + \lambda \norm{x}_1 + \qfrac{\alpha}{2} \norm{x}_2^2,
\eeq
where $A\in\bbR^{m\times n}$, $\lambda,\ \alpha>0$. The term $\frac{\alpha}{2} \norm{x}_2^2$ is introduced to ensure strong convexity of the objective function such that Forward-Backward splitting method enjoys global linear convergence for both $\seq{\xk}$ and $\seq{\uk}$. 

Based on Proposition \ref{prop:step} and Example \ref{eg:FB-step}, we can derive an upper bound for the identification step. To validate this, we consider a problem with dimensions $(m,n)=(20,32)$ and run the  method with stopping criterion $\norm{\xk-\xkm}<10^{-15}$. Figure \ref{fig:support-FBS} illustrates the evolution of the support of $\xk$ over iterations. 
The location of the {\it red dashed} line is our estimation of the number of steps for identification. Note that our upper bound is quite loose which is the consequence of the inequalities involved in the derivation. 
Nonetheless, the experiments validates our result in Proposition \ref{prop:step}.

\begin{figure}[H]
    \centering
    \includegraphics[width=0.45\linewidth]{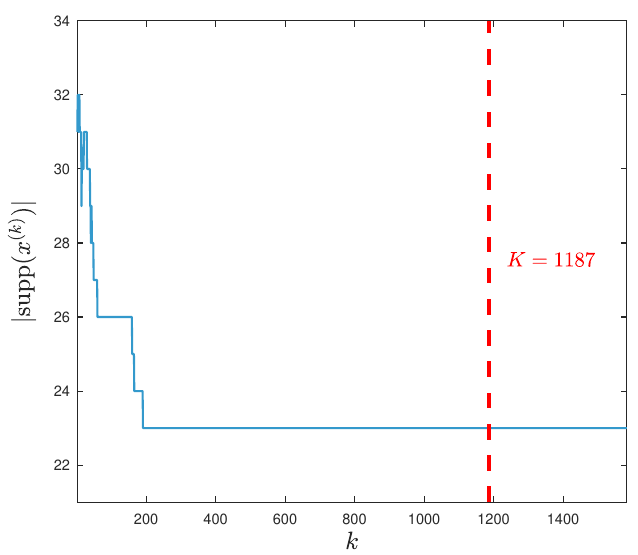}
    \caption{Identification of $\xk$ and the estimated identification step. The red dashed line represents the upper bound for the number of steps required for identification.}
    \label{fig:support-FBS}
\end{figure}

\section{Conclusion}\label{sec:conclusion}

In this paper, we proposed the definition of partly smooth operators as a complement to \cite{lewis2022partial}, offering a new perspective on partial smoothness with a geometric explanation. Under the framework of partial smoothness, we demonstrated the identification property under more relaxed conditions, specifically for non-convergent dual sequences and no degenerate assumptions. 
We conducted numerical experiments to illustrate our theory, showing that identification can occur even when the dual vector is not convergent. For instance, in the mini-batch stochastic gradient descent (SGD) example, identification still takes place when the batch size is sufficiently large, indicating that the bounded distance of dual vector is enough to promise identification.
Furthermore, under the local union structure, we revealed additional noteworthy results, including an estimation of the upper bound for the identification step. While this has been mentioned in the literature such as \cite{liang2017activity,liang2014convergence}.
As long as the operator exhibits local smoothness and the algorithm satisfies the appropriate properties, we can derive a suitable estimation.

\begin{acknowledgements}
JL is supported by the National Natural Science Foundation of China (No. 12201405), the ``Fundamental Research Funds for the Central Universities'', the National Science Foundation of China (BC4190065) and the Shanghai Municipal Science and Technology Major Project (2021SHZDZX0102).
\end{acknowledgements}

\section*{Conflict of interest}

The authors declare that they have no conflict of interest.

\bibliographystyle{spmpsci}      % mathematics and physical sciences
\bibliography{psso}

% \begin{thebibliography}{}
% %
% % and use \bibitem to create references. Consult the Instructions
% % for authors for reference list style.
% % Format for Journal Reference
% Author, Article title, Journal, Volume, page numbers (year)
% Author, Book title, page numbers. Publisher, place (year)
% \end{thebibliography}

\end{document}